\documentclass[risks,article,submit,moreauthors,pdftex,10pt,a4paper]{mdpi}

\firstpage{1}
\articlenumber{x}
\doinum{10.3390/------}
\pubvolume{xx}
\pubyear{2016}
\copyrightyear{2016}
\externaleditor{Academic Editor: name}
\history{Received: date; Accepted: date; Published: date}
\pdfoutput=1

\usepackage{graphicx}
\usepackage{amsmath}
\usepackage{setspace}
\usepackage[english]{babel}
\usepackage{geometry}
\usepackage{amsmath,amssymb,amsthm,bbm,color,graphics,version}
\usepackage{mathrsfs}
\usepackage{relsize}
\usepackage{subcaption}

\newcommand{\bearno}{\begin{eqnarray*}}
\newcommand{\enarno}{\end{eqnarray*}}
\newcommand{\beeq}{\begin{equation}}
\newcommand{\eneq}{\end{equation}}

\newcommand{\R}{\mathbb{R}}

\newcommand{\expect}[1]{{\mathbb E}\left[#1\right]}
\newcommand{\pr}[1]{{\rm {\mathbb P}}(#1)}

 \theoremstyle{mdpi}
 \newcounter{thm}
 \newcounter{ex}
 \newcounter{re}
 \newtheorem{Theorem}[thm]{Theorem}
 \newtheorem{Lemma}[thm]{Lemma}
 \newtheorem{Corollary}[thm]{Corollary}

 \theoremstyle{mdpidefinition}

 \newtheorem{Example}[ex]{Example}
 
 \newtheorem{Remark}[re]{Remark}
 \newtheorem{Definition}[thm]{Definition}

\Title{Ruin probabilities with dependence on the number of claims within a fixed time window}

\Author{Corina Constantinescu \texorpdfstring{$^{1}$}{1}, Suhang Dai \texorpdfstring{$^{1}$}{1}, Weihong Ni \texorpdfstring{$^{1}$}{1}* and Zbigniew Palmowski \texorpdfstring{$^{2}$}{2}}
\AuthorNames{Firstname Lastname, Firstname Lastname and Firstname Lastname}

\address{%
\texorpdfstring{$^{1}$}{1} \quad Institute for Financial and Actuarial Mathematics, Department of Mathematical Sciences, University of
Liverpool,  Liverpool L69 7ZL, UK\\
\texorpdfstring{$^{2}$}{2} \quad Mathematical Institute, University of Wroc{\l}aw, Poland}

\corres{W.Ni@liverpool.ac.uk; Tel.: +44 787 156 7797}

\abstract{
We analyse the ruin probabilities for a renewal insurance risk process with inter-arrival time distributions depending on the claims that arrived within a fixed (past) time window. This dependence could be explained through a regenerative structure. The main inspiration of the model comes from the Bonus-Malus feature. We discuss first asymptotic results of ruin probabilities for different regimes of claim distributions. For numerical results, we recognise an embedded Markov additive process. Via an appropriate change of measure, ruin probabilities could be computed to a closed form formulae. Additionally, we present simulated results via the importance sampling method, which further permit an in-depth analysis of a few concrete cases.
}


\begin{document}
\nolinenumbers


\section{Introduction}\label{sec:intr}
With the ever growing popularity of Bonus-Malus systems, one interesting question to study would be whether it really reduces the associated risk and with how much. A common measure to assess risks an insurer is exposed to is via the so-called ruin probabilities. Motivated by such kind of questions, we try to compute the probability of ruin for a simple Bonus system (also called no claim discount (NCD) system) in this paper. The main feature of such systems is that there is a premium discount when no claims are observed in the previous year. Inspired by this feature, we found a regenerative structure within inter claim times that could describe the dependence in a Bonus system equivalently.\\

For the simplest case, there are only two classes - either a base or a discounted level in the NCD system under the consideration here. The discounted level implies a lower premium rate and occurs when no claim is witnessed in a past fixed time window $\xi$. That is to say, the portfolio moves between these two classes. The switching condition relies on the history of arrived claims within the fixed time window $\xi$. Theoretically speaking, it also works for a merely Malus system. Yet in practice, such systems do not exist as it probably sounds more tempting if an insurance company offers rewards rather than penalties. Therefore, the incorporation of such dependence in a risk model violates some of the classical assumptions, thus making it more difficult to calculate ruin probabilities. However, by equating the dependence between claim arrivals and premium rates with that between two consecutive inter arrival times (Figure \ref{fig:BMS}), a regenerative structure can be identified so that further analysis could be carried out.\\

Looking into literature, one extension from a classical risk model is to relax the assumption of independence. Hence, dependence modelling has been introduced under a risk theory framework. There are several kinds of dependence to be considered. For a dependence within claims, Albrecher et al. \cite{albrecher2011explicit} calculated ruin probabilities by using Archimedean survival copulas. Through a copulas method, Valdez and Mo \cite{valdez2002ruin} also worked with risk models with dependence among claim occurrences focusing more on the simulation side. The dependence between claim sizes and inter-arrival times was also analysed. Albrecher and Boxma \cite{albrecher2004ruin} first considered the case when the inter-claim time depends on the previous claim size with a random threshold. Due to the complexity of inverting Laplace Transform, they could only obtain the results in terms of Laplace Transforms. Kwan and Yang \cite{kwan2010dependent} then studied such a dependence structure with a deterministic threshold and they were able to express ruin probability explicitly by solving a system of ordinary delay differential equations.\\

On the other hand, most of the work on Bonus-Malus systems mainly relied on a construction of a discrete Markov Chain in order to compute different levels of prices. See \cite{Lemaire}. Due to the possibility of slow convergence to stationarity, \cite{asmussen2014} added an age-correction and implemented numerical analyses for various Bonus-Malus systems in different countries. Instead of considering premium levels depending only on claims in the previous period, it is alternatively suggested to take into account of the entire history where a Bayesian view could be adopted as in \cite{asmussen2014}.  A recent work by Ni et al. \cite{ni2014} also applied the Bayesian approach to reflect this idea and obtained premiums in a closed form when claim severities are assumed to be Weibull distributed.\\

There have been a few papers investigating ruin probabilities for a Bonus-Malus system recently. Working with real data, \cite{Alfredo1, Alfredo2} calculated ruin probabilities under a realistic Bonus-Malus framework. The idea in their work is that they first analyse ruin probabilities for a single year conditioning on the reserve levels at the beginning and the end of the year. Since premium rate is kept constant within a year, a classical technique could be borrowed. Then they use approximations and estimate ruin probabilities numerically. On the other hand, the incorporation of the feature of a Bonus-Malus system into a risk model is similar to a variation in the premium rates after a claim arrival. For such a dependence, by employing a Bayesian estimator in a risk model and using a comparison method with the classical case, Dubey \cite{Dubey} and Li et al. \cite{Li2015} could interpret ruin probabilities in terms of a classic one. In this paper, we follow a similar idea. However, our work here is to model the dependence between premium rates and the claim arrivals within a fixed time window. In our opinion it mimics better the key feature of a Bonus system. The consequence of this approach is that it implies modelling the dependence between two consecutive inter-claim times based on a fixed threshold, which is explained by Figure \ref{fig:BMS}. This serves as the main goal of this work. Furthermore, the allowance of such dependence obviously violates the renewal property as in the classical model. Nevertheless, a regenerative structure can be identified so that anlayses are possible. For literature on regenerative processes, we refer to \cite{BP1, BP2}.\\

\begin{center}
\begin{figure}[!htb]
\scalebox{0.3}[0.3]{\includegraphics{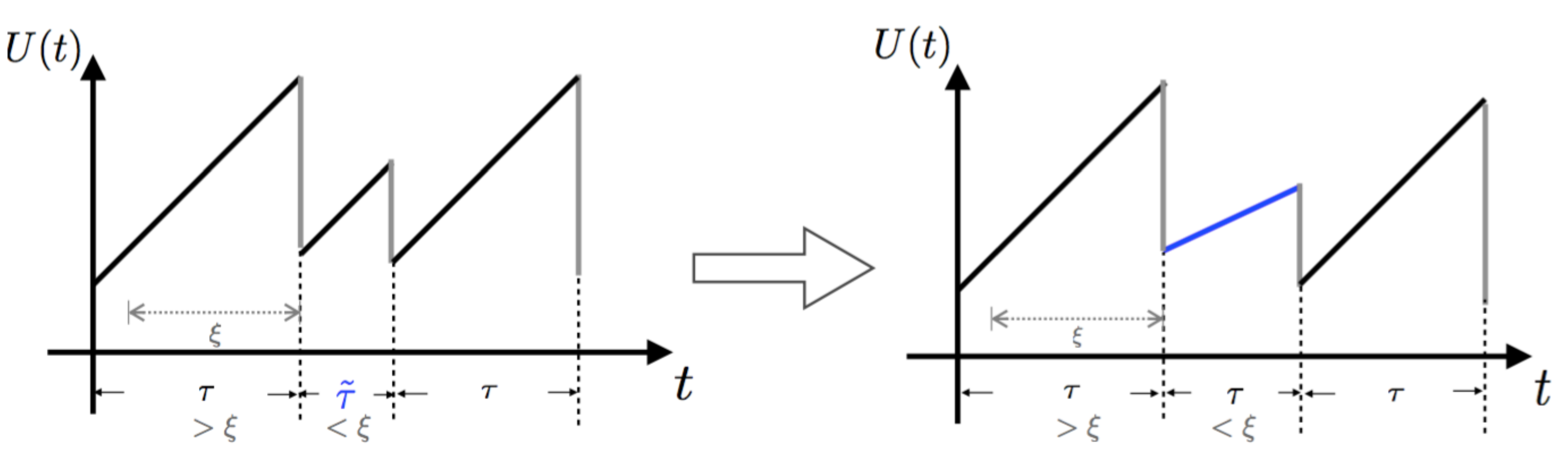}}
\caption{Model transform}
\label{fig:BMS}
\end{figure}
\end{center}
In Figure \ref{fig:BMS}, Let us denote the inter claim time by $\tau$. The graph on the right shows a two-level Bonus system where the premium rate decreases after a relatively long wait which exceeds a fixed number $\xi$. In reality, this fixed window $\xi$ could be understood as a calendar year for instance, because many insurances companies charge different premiums based only on yearly claim histories. After that, since the second waiting interval is less than $\xi$, the premium rate returns to its original value and so on and so forth. Equivalently, this could be transferred to a model where the adjustment on premium rates is reflected in inter arrival times switching between two different random variables, as long as the increment of the surplus process $U(t)$ in this time interval is kept the same. That is to say, whenever a large inter arrival time which is above $\xi$ is witnessed, the next inter claim time will switch to a random variable $\tilde{\tau}$ with a different distribution. As we work with the ruin probabilities under an infinite-time horizon, such transformation would not affect the results. $\tilde{\tau}$ could be assumed to have a smaller mean than $\tau$ for a more 'realistic' interpretation, resulting from the effect of the drop in the premium rate. Additionally, for computational reasons, we make the assumption that the inter-exchange of the randomness of inter arrival times only happens after a jump rather than precisely at the end of the fixed window.\\

Aiming at studying the ruin probability of a Bonus system, we try to investigate this model under a regenerative framework using various methods. We start by looking at some asymptotic results via adopting theories developed for general regenerative processes in \cite{BP1, BP2}. For the Cram\'{e}r case, it still shows exponential tails for the probability of ruin. All asymptotic results are shown for a general situation where the distribution for each random variable is not specified. Then, by constructing an appropriate Markov Additive Process and using the Importance Sampling method we can run simulations to get numerical results for the case where inter claim times and claims are exponentially distributed. In the end, we employ the crude Monte Carlo simulations to compare the underlying ruin probability with a classic one as a case analysis. In addition, we look at the influence of claim distributions on the ruin probabilities, where we found that there is no significant differences in ruin probabilities when altering between Exponential and Pareto claims. In general, results suggest that the use of Bonus systems may not act in favour of the reduction on ruin probabilities. That is probably because premium discounts generally decrease the risk reserves. However, if the insurer is able to gain more market share by providing a Bonus system, although ruin probabilities could not be improved, their revenues and profits could possibly experience a positive effect.\\

\section{The model}
Let us start from describing the model that we will work in this paper with.
We denote by $U(t)$ the amount of surplus  of an insurance portfolio at
time $t$:
\begin{equation}
\label{eq1}
U(t) = x +c\,t -\sum_{k=1}^{N(t)} Y_k.
\end{equation}
In the above classical model \eqref{eq1}, $c$ represents the constant rate of premiums inflow, $N(t)$ is a arrival process
that counts the number of claims incurred during the time interval
$(0,t]$ and $(Y_k)_{k\geq 0}$ is a sequence of independent and identically distributed (i.i.d.) claim sizes with distribution function $F_Y$ and density $f_Y$
(also independent of the claim arrival process $N(t)$). We assume that $U(t)\to+\infty$ a.s. as $t\to+\infty$. One of the crucial quantities to investigate in this context is the probability that the surplus in the portfolio will not be sufficient to cover the claims for the first time, which is called the probability of ruin
\[
\psi(x) = \pr{T(x)<\infty \mid U(0)=x}.
\]
Here $U(0)=x\geq 0$ is the initial reserve in the portfolio and $$T(x)=\inf \:\{t\geq 0:\,U(t)<0 \mid U(0)=x\}$$
is the time of ruin for an initial surplus $x$. We specify the counting process $N(t)$ which in the classical models is a Poisson process.
Let $(\tau_k)_{k\geq 0}$ be the sequence of inter-claim times. In this paper we analyse the model when the distribution $F_{\tau_k}$ of $\tau_k$
depends on the number of claims that appeared within a fixed past time window $\xi$ as follows,
$$\mathbb{P}(\tau_k\leq x)=F_{\tau_k}\left(x, N\left(\sum_{i=1}^{k}\tau_{i-1}\right)-N\left(\sum_{i=1}^{k}\tau_{i-1}-\xi\right)\right).$$
It is true that when such dependence structure is introduced, a direct use of renewal theory is no longer applicable here. However, taking a second look, even though it is not renewal at each jump epoch, the process in fact renews after several jumps and we call this a 'regeneration'. Thus, we define the regenerative epochs for our model in the following way.
\begin{Definition}
Regeneration epochs $T_{k+1}, k=0, 1, \ldots$ are defined as
$$T_{k+1}=\min\left\{\tau_l\geq T_k: N\left(\sum_{i=1}^{l}\tau_i\right)-N\left(\sum_{i=1}^{l}\tau_i-\xi\right)=0\right\}$$ with $T_0=0$.\label{regtime}
\end{Definition}
Roughly speaking, at these epoch $T_k$ (being the arrival times with zero number of arrivals within the last time window lagged by $\xi$)
the risk process $U(t)$ loses his 'memory' and starting at these epochs the stochastic evolution remains the same.
A formal definition of a regenerative process can be found in Appendix A1 in \cite{RP}. It easy to observe that
the risk process $U(t)$ is indeed regenerative with regeneration epochs $T_k$. In this case, $\mathbb{P}(\tau_k\leq x)=F_{\tau_k}(x,0)$. Notice that we define the regenerative epochs in such a way that the concern only lies in whether there are claims or not in the past fixed window $\xi$ rather than how many of them.\\

Moving into details, let us consider the claim surplus process denoted by
$$S(t)=\sum_{k=1}^{N(t)}Y_i-c.t$$
Moreover, let
\[
M   = \sup_{t\geq 0} S(t), \hspace{1cm} M_{n+1} = \sup_{T_n\leq t<
T_{n+1}} S(t)-S(T_n), \quad\text{for $n\geq 0$},
\]
and
\begin{equation}\label{eqn:XM}X_{n+1}=S(T_{n+1})-S(T_n).\end{equation}
Then due to the regenerative structure of the claim surplus process $S(t)$, the discrete-time process
$S_n=S(T_n)=\sum_{i=1}^nX_i$ ($n\geq 0$) is a random walk.
The crucial observation used in this paper is:
 \begin{equation}\label{numer}
\psi(x)=\pr{M>x}= \pr{\max_n (M_n+S_{n-1})>x}.\end{equation}

The simplest case that we focus on is the one with inter claim times following two random variables $\tau$ and $\tilde{\tau}$. The first one $\tau$ we choose when in a past time-window of length $\xi$ there is at least one claim. Otherwise we choose $\tilde \tau$ as the inter arrival time. Hence
\begin{equation*}
\mathbb{P}(\tau_k\leq y)=\left\{\begin{array}
{l@{\quad\quad}l}
\mathbb{P}(\tau\leq y),\;  \mbox{if} \;\;N(\sum_{i=1}^{k}\tau_{i-1})-N(\sum_{i=1}^{k}\tau_{i-1}-\xi)\geq 1,\\
\mathbb{P}(\tilde{\tau} \leq y),\; \mbox{otherwise}. \end{array} \right.\nonumber
\end{equation*}
It is a natural choice since usually in insurance company a long "silence" translates into a different behaviour of the arrival process just right after it. To rephrase it, our current model incorporates a dependence structure between each pair of consecutive inter-arrival times. Whenever an inter-arrival time exceeds $\xi$, the next one would have the same distribution as $\tilde{\tau}$. Otherwise, it conforms to $\tau$.\\

More interestingly, as mentioned earlier, such model set-up would fit into a basic Bonus system, i.e., a system where policyholders enjoy discounts when they do not file claims for a certain period (but with no penalties). Without loss of generality, Figure \ref{fig:BMS} plots an example of such risk processes and demonstrate how our model reflects the feature of a Bonus system.\\


An example of sample path of the claim surplus process we will be working with is given in Figure \ref{fig:claimsurplus} (where we assume starting from $\tilde{\tau}$).
\begin{figure}[!htb]
\centering
\scalebox{0.25}[0.25]{\includegraphics{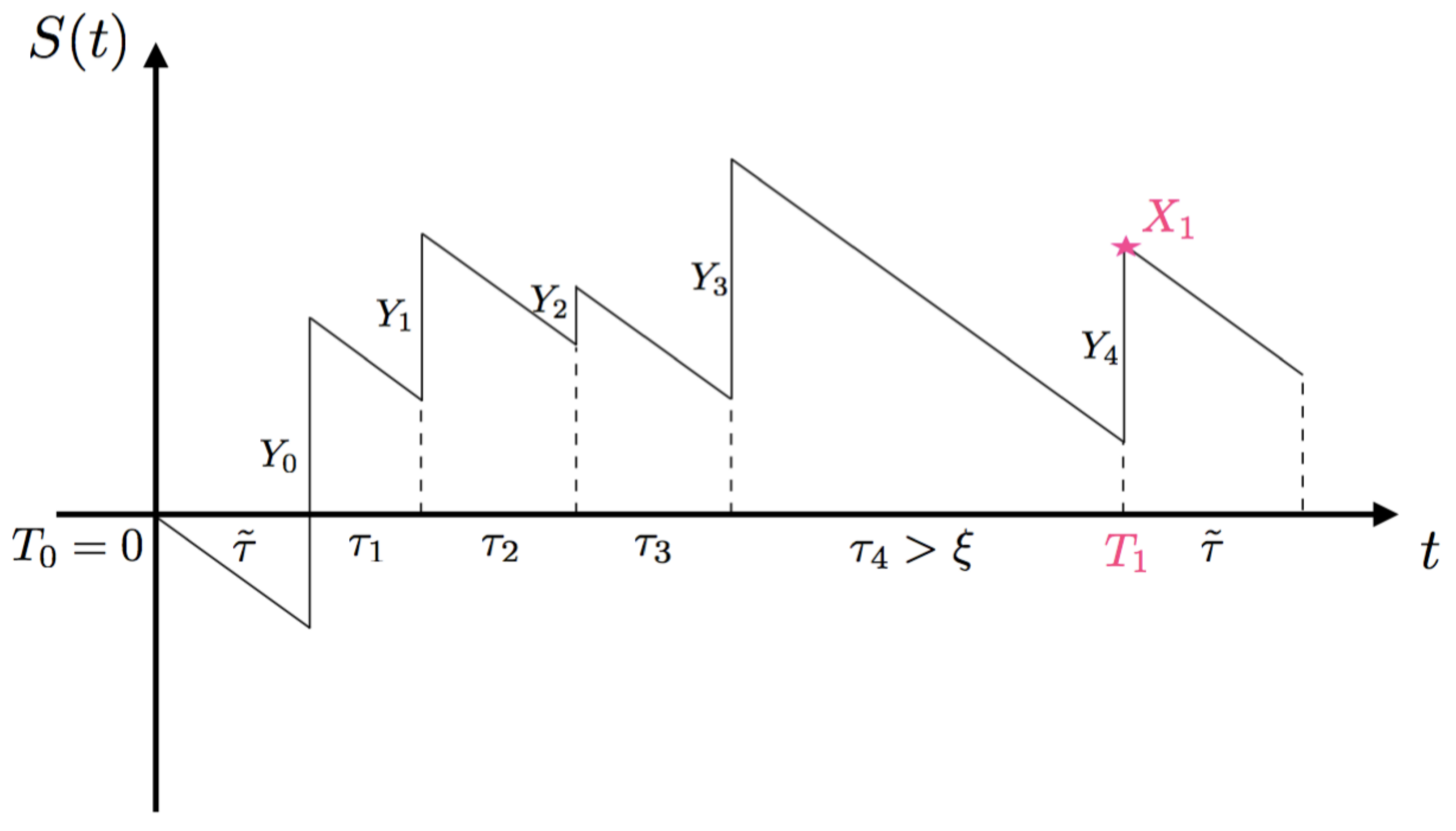}}
\caption{A sample path of the regenerative process}
\label{fig:claimsurplus}
\end{figure}
Recall from \eqref{eqn:XM} that $X_1$ is the end value at the first regenerative epoch. Then it is not difficult to observe that it has the same law as
\begin{equation}\label{reprx1}
X_1\overset{d}{=}(Y_0-\tilde{\tau})+\mathbf{I}_{\{\tilde{\tau}\leq \xi\}}\left(\sum_{k=1}^{N-1}(Y_k-\tau_k^{\leq\xi})+\left(Y_{N}-\tau^{>\xi}_{N}\right)\right),
\end{equation}
where $N$ is a geometrical random variable with parameter $p=\mathbb{P}(\tau>\xi)$. Here $\mathbb{P}(N=k)=(1-p)^{k-1}p$, $k=1,2\ldots$ and $\expect{\tau_k^{\leq\xi}}=\mathbb{E}[\tau_k|\tau_k\leq \xi]$, $\expect{\tau_k^{>\xi}}=\mathbb{E}[\tau_k|\tau_k> \xi]$.\\

The paper is organized as follows. Section 3 presents asymptotic results about ruin probabilities under three different regimes for claim distributions, using asymptotics derived for general regenerative processes as in \cite{BP1, BP2}.  Section 4 demonstrates some numerical results via simulations and discusses a case analysis including comparison with ruin in a classical risk model. The simulations used in this section are based on the embedded Markov additive process within our model and rely on the importance sampling method via a change of measure. Our work will be concluded in Section 5.

\section{Asymptotic results}\label{main}
In this section, we look at three different situations for claim distributions and analyse the asymptotic ruin probability associated with each of them. Inter arrival times considered in this section are general random variables if not mentioned specifically.

\subsection{The heavy-tailed case}\label{heavy}
Let us first discuss the heavy-tailed case. We start with the assumption that the distribution $F_M$ of generic $M_1$
belongs to the class~$\mathcal{S}$
of \emph{subexponential} distribution functions, where a distribution
function~$G\in\mathcal{S}$ on $\R_+$ if and only if
$\overline{G}(x)>0$, for all $x$, and
\begin{equation}\label{eq:5}
  \lim_{x\to\infty}\overline{G^{*2}}(x)/\overline{G}(x) = 2
\end{equation}
(where $G^{*2}$ is the convolution of $G$ with itself).
Here $\overline{G}$ denotes the tail distribution given by
$\overline{G}(x)=1-G(x)$.
More
generally, a distribution function~$G$ on $\R$ is subexponential if
and only if $G^+$ is subexponential, where $G^+=G\mathbf{I}_{\R_+}$ and
$\mathbf{I}_{A}$ is the indicator function of a set $A$.
We further assume throughout that $F_M \in \mathcal{S}^*$ are strong subexponential distributions.
According to Definition 3.22 in \cite{FKZ}, a distribution function~$G$ on $\R$ belongs to the class~$\mathcal{S}^*$, i.e., G is strong subexponential, if and
only if $\overline{G}(x)>0$, for all $x$, and
\begin{equation}
  \label{eq:6}
  \int_0^x\overline{G}(x-y)\overline{G}(y)\,dy
  \sim 2m_{G}\overline{G}(x),\quad\text{as $x \to \infty$},
\end{equation}
where
\begin{displaymath}
  m_{G} = \int_0^\infty \overline{G}(x) \,dx
\end{displaymath}
is the mean of $G$.  It is again known that the property $G\in\mathcal{S}^*$
depends only on the tail of $G$.  Further, if $G\in\mathcal{S}^*$ then $G\in\mathcal{S}$
and also $\overline{G^s}\in\mathcal{S}^*$ where
\begin{displaymath}
  \overline{G^s}(x) = \min \biggl( 1,
    \int_x^{\infty} \overline{G}(t)\, dt
  \biggr).
\end{displaymath}
is the integrated, or \emph{second-tail}, distribution function
determined by $G$. See \cite{FKZ} for details.

\begin{Theorem}
\label{infinitezerodelay}
If $\expect{M_1}<\infty$ and $F_M\in \mathcal{S}^*$ then
\beeq
\label{infinitetime}
\psi(x)\sim \frac{1}{\mu} \int_x^\infty\overline{F}_M(u) {\rm d}u,
\eneq
as $x\rightarrow\infty$, with $\mu=-\expect{X_1}$.
\end{Theorem}

Note that
$$\pr{X_1>x}\leq \pr{M_1>x}\leq \pr{X_1+T_1>x}.$$
Assume now that $\tau$ and $\tilde{\tau}$ are light-tailed, that is
there exists $\theta >0$ such that $\expect{{{\rm e}}^{\theta \tau}}<\infty$ and $\expect{{{\rm e}}^{\theta \tilde{\tau}}}<\infty$
and that
\begin{equation}\label{asYsub}
F_M\in \mathcal{S}^*.
\end{equation}
Then from Foss et al. \cite{FKZ} we have that
\begin{equation}\label{eqxm}
\pr{M_1>x}\sim \pr{X_1>x}
\end{equation}
and by (\ref{reprx1}) and Corollary 3.40 in \cite{FKZ} we have that
\begin{eqnarray}\label{asx}
\overline{F}_M(x)\sim\pr{X_1>x}&\sim& \left(\pr{\tilde{\tau}>\xi}+\expect{N+1}\pr{\tilde{\tau}\leq \xi}\right)
\pr{Y>x}\nonumber\\
&=&\left(1-\pr{\tilde{\tau}\leq \xi}+\frac{\pr{\tilde{\tau}\leq \xi}(2-\pr{\tau\leq \xi})}{1-\pr{\tau\leq \xi}}\right)\overline{F}_Y(x),
\end{eqnarray}
where $Y$ is a generic claim size.
Moreover,
\begin{eqnarray*}
\mu &=& (\expect{\tilde{\tau}}-\expect{Y})+\left[\expect{N-1}\left(\expect{\tau^{\leq\xi}}-\expect{Y}\right)+(\expect{\tau^{>\xi}}-\expect{Y})\right]\pr{\tilde{\tau}\leq \xi}\\
&=& \expect{\tilde{\tau}}-\expect{Y}-\pr{\tilde{\tau}\leq \xi}\expect{N}\expect{Y}+\expect{N-1}\expect{\tau|\tau\leq\xi}\pr{\tilde{\tau}\leq \xi}+\expect{\tau|\tau>\xi}\pr{\tilde{\tau}\leq \xi}
\end{eqnarray*}
with
$$\expect{N}=\frac{1}{1-\pr{\tau\leq \xi}}=\frac{1}{\pr{\tau> \xi}}.$$

\begin{Remark}\rm{\bf Reducing to the classical model}\\
Removal of the dependence in our setting would reduce to the classical model. An independent case is referring to the situation when $\pr{\tau\leq\xi}=\pr{\tilde{\tau}\leq\xi}$. Substituting this into \eqref{asx} yields,
\begin{equation*}
\overline{F}_{M}(x)\sim \frac{1}{1-\pr{\tau\leq\xi}}\overline{F}_{Y}(x)=\expect{N}\overline{F}_{Y}(x).
\end{equation*}
In addition, it simplifies $\mu$ to
\begin{equation*}
\mu = -\expect{X_1} = \expect{N}(\expect{\tau}-\expect{Y}).
\end{equation*}
According to (\ref{infinitetime}),
\begin{equation*}
\psi(x)\sim \frac{\expect{N}}{\mu} \int_x^\infty\overline{F}_Y(u) {\rm d}u=\frac{1}{\expect{\tau}-\expect{Y}}\int_x^\infty\overline{F}_Y(u) {\rm d}u.
\end{equation*}
If we assume $\expect{Y}=\mu_Y$, $\expect{\tau}=\expect{\tilde{\tau}}=\frac{1}{\lambda}$ and the safety loading to be $\rho$ i.e., $1+\rho=1/\lambda\mu_Y$. Also, we know that $\expect{N}=1/p$. The above identity could be reduced to,
\begin{equation}
\psi(x)\sim\frac{1}{\frac{1}{\lambda}-\mu_Y}\int_x^\infty\overline{F}_Y(u) {\rm d}u=\frac{1}{\rho\mu_Y}\int_x^\infty\overline{F}_Y(u) {\rm d}u,
\end{equation}
which coinsides with the approximated ruin probability for the classic risk model with subexponential claims as shown by Theorem 1.36 in \cite{embrechts1997}.
\end{Remark}

\subsection{The intermediate case}

We now consider the case where  $X_1$ satisfies \beeq
\label{l-alpha} \pr{X_1>x+y} \sim {{\rm e}}^{-\alpha y} \pr{X_1>x}
\eneq for every fixed $y$, as $x\rightarrow\infty$, and \beeq
\label{s-alpha} \pr{X_1+X_2>x} \sim 2 \expect{{{\rm e}}^{\alpha
X_1}} \pr{X_1>x}. \eneq This is equivalent to the condition that
$X_1^+ \in {\mathcal S}(\alpha)$. The case $\alpha=0$ is treated
in Section 2.1 so we assume $\alpha>0$. Another assumption we make
is that \beeq \label{nocramer} \expect{{{\rm e}}^{\alpha X_1}}<1,
\eneq which implies that Cram\'er condition is not satisfied.
Finally, we specify the tail behavior of $M_1$. The case where
$M_1$ has a heavier tail than $X_1$ is already covered in the
previous subsection. Motivated by this, we assume that
\begin{equation}
\label{tailbalance}
\lim_{x\rightarrow\infty} \frac{\pr{M_1>x}}{\pr{X_1>x}} <\infty
\end{equation}
(we allow the limit to equal 0).
Furthermore, we assume that there exists a bounded
function $g$ such that
\beeq \label{cough}
\lim_{x\rightarrow\infty} \frac{\pr{M_1>x ; X_1 \leq
x-a}}{\pr{X_1>x}} = g(a), \eneq for all real values of $a$.

\begin{Theorem}\label{thmsalpha}
Suppose that (\ref{l-alpha})--(\ref{cough}) are satisfied.
Then
\[
\psi(x) \sim \frac{\expect{{{\rm e}}^{\alpha M}} + \expect{g(M)}}{1-\expect{{{\rm e}}^{\alpha X_1}}}\pr{X_1>x}.
\]
\end{Theorem}

\begin{proof}
Proof can be found in \cite{BP1}.
\end{proof}

We discuss now when conditions \eqref{l-alpha}-\eqref{cough} are satisfied.
Assume that $$F_Y\in \mathcal{S}(\alpha)$$
that is $F_Y$ satisfies assumptions (\ref{l-alpha}), (\ref{s-alpha}) and (\ref{nocramer}).
Then by representation (\ref{reprx1}) we can easily check that $X_1$ also satisfies
\begin{equation}\label{equivsalpha} \pr{X_1>x}\sim D(\alpha)\overline{F}_Y(x)
\end{equation}
and
\begin{equation}\label{Dalpha}
D(\alpha)=\expect{{{\rm e}}^{-\alpha \tilde{\tau}}}+\pr{\tilde{\tau}\leq \xi}\left(\expect{N\expect{{{\rm e}}^{\alpha
(Y-\tau^{\leq \xi})}}^{N-1}}\expect{{{\rm e}}^{-\alpha \tau^{\leq \xi}}}+\expect{{{\rm e}}^{-\alpha \tau^{>\xi}}}\right).\end{equation}
We additionally assume that $Y$ satisfies
$$ \expect{{{\rm e}}^{\alpha
X_1}}=\varphi(\alpha)<1;$$
(see \eqref{eqn:mgfx2} for the representation of $\varphi(\theta)$).
Note that the assumption (\ref{tailbalance}) is always satisfied in our case since similarly to (\ref{equivsalpha}) we have:
$$\lim_{x\rightarrow\infty} \frac{\pr{M_1>x}}{\pr{X_1>x}}\leq  \lim_{x\rightarrow\infty} \frac{\pr{X_1>x-T_1}}{\pr{X_1>x}}
\leq\frac{\expect{N\left(\expect{{{\rm e}}^{\alpha
Y}}\right)^{N-1}}+2}{D(\alpha)}<\infty.$$
Finally, conditioning on $N$, using representation (\ref{reprx1}) and property (\ref{l-alpha}) gives:
$$g(a)= \frac{1}{\expect{{{\rm e}}^{-\alpha \Xi}}}-{{\rm e}}^{-\alpha a}\pr{\Xi>a}$$
for $\Xi=\tilde{\tau}+\left(\tau^{>\xi}+\sum_{k=1}^{N^*}\tau^{\leq \xi}_k\right)\mathbf{I}_{\{\tilde{\tau}\leq \xi\}}$, where
\begin{eqnarray*}
\lefteqn{N^*=\min\left\{n\leq N:  \mathbf{I}_{\{\tilde{\tau}\leq \xi\}}\left(\sum_{k=1}^{n}(Y_k-\tau_k^{\leq\xi})\right)\right.}\\&&\qquad=\left.
\max_{l\leq N}\left[\mathbf{I}_{\{\tilde{\tau}\leq \xi\}}\left(\sum_{k=1}^l(Y_k-\tau_k^{\leq\xi}) \right)\right]\right\}.
\end{eqnarray*}

\subsection{The Cram\'er case}\label{cs}
In this subsection, we review the extension of the classical Cram\'er case from random walks to perturbed random walks
and regenerative processes.

\begin{Theorem}\label{thm:cramer}
Assume that there exists a solution $\kappa>0$ to the equation
\[
\expect{{{\rm e}}^{\kappa X_1}}=1 \mbox{ such that } m=\expect{X_1
{{\rm e}}^{\kappa X_1}}<\infty.
\]
Assume furthermore that $X_1$ is non-lattice and that $\expect{{{\rm e}}^{\kappa M_1}}<\infty$.
Then
\[
\psi(x) \sim  K {{\rm e}}^{-\kappa x}
\]
with $K=\frac 1{\kappa m} \expect{{{\rm e}}^{\kappa M_1}- {{\rm e}}^{\kappa (M+X_1)} ; M_1 > M+X_1}$ for independent $M$ of $X_1$ and $M_1$.
\end{Theorem}

\begin{proof}
See \cite{goldie}.
\end{proof}

It is easy to see that $K$ is bounded from above by
\begin{equation}\label{Cramerconstant0}
\bar K=\expect{{{\rm e}}^{\kappa M_1}}/(\kappa m).
\end{equation}
In fact it is even bounded above by
\begin{equation}\label{Cramerconstant}
\tilde K=\expect{{{\rm e}}^{\kappa (X_1+T_1)}}/(\kappa m).
\end{equation}

Note that by (\ref{reprx1}) the Cram\'er adjustment coefficient $\kappa>0$ solves
\[\expect{{{\rm e}}^{\kappa X_1}}=\varphi(\kappa)=1\]
for
\begin{eqnarray}
\expect{{{\rm e}}^{\kappa X_1}}&=&\tilde{p}\expect{{{\rm e}}^{\kappa Y}}\expect{{{\rm e}}^{-\kappa \tilde{\tau}}|\tilde{\tau}>\xi}
+\tilde{q}\expect{{{\rm e}}^{\kappa Y}}\expect{{{\rm e}}^{-\kappa \tilde{\tau}}|\tilde{\tau}\leq\xi}\cdot\expect{{{\rm e}}^{\kappa Y}}\expect{{{\rm e}}^{-\kappa \tau}|\tau>\xi}\nonumber\\
&&\cdot\sum_{k=1}^\infty p(1-p)^{k-1}\left[\expect{{{\rm e}}^{\kappa Y}}\expect{{{\rm e}}^{-\kappa \tau}|\tau\leq\xi}\right]^{k-1}
\nonumber\\
&=&p\tilde{q}\frac{(\expect{{{\rm e}}^{\kappa Y}})^{2}\expect{{{\rm e}}^{-\kappa \tau}|\tau>\xi}\expect{{{\rm e}}^{-\kappa \tilde{\tau}}|\tilde{\tau}\leq\xi}}{1-(1-p)\expect{{{\rm e}}^{\kappa Y}}\expect{{{\rm e}}^{-\kappa \tau}|\tau\leq\xi}}+\tilde{p}\expect{{{\rm e}}^{\kappa Y}}\expect{{{\rm e}}^{-\kappa \tilde{\tau}}|\tilde{\tau}>\xi},
\label{eqn:mgfx1}
\end{eqnarray}
where $\pr{\tilde{\tau}> \xi}=\tilde{p}, \,\pr{\tilde{\tau}\leq \xi}=1-\tilde{p}=\tilde{q}$.
The above calculations shows also that the m.g.f. $\varphi$ of $X_1$ has the following representation:
\begin{eqnarray}
\varphi(\theta)&=&p\tilde{q}\frac{(\expect{{{\rm e}}^{\theta Y}})^{2}\expect{{{\rm e}}^{-\theta \tau}|\tau>\xi}\expect{{{\rm e}}^{-\theta \tilde{\tau}}|\tilde{\tau}\leq\xi}}{1-(1-p)\expect{{{\rm e}}^{\theta Y}}\expect{{{\rm e}}^{-\theta \tau}|\tau\leq\xi}}+\tilde{p}\expect{{{\rm e}}^{\theta Y}}\expect{{{\rm e}}^{-\theta \tilde{\tau}}|\tilde{\tau}>\xi}.
\label{phitheta}
\end{eqnarray}
We can now identify constant $\tilde K$:
\begin{eqnarray}
\tilde{K}&=&\expect{{\rm e} ^{\kappa(X_1+T_1)}}/\kappa m\nonumber=\expect{{\rm e} ^{\kappa \sum_{i=1}^{N(T_1)}Y_i}}/\kappa m\nonumber\\
&=&\frac{1}{\kappa m}\sum^{\infty}_{n=1} \left(\expect{{\rm e} ^{\kappa Y}}\right)^n \pr{N=n}\nonumber\\
&=&\left(\pr{\tilde{\tau}>\xi}\expect{{\rm e} ^{\kappa Y}}+\pr{\tilde{\tau}\leq\xi}\sum^{\infty}_{n=2}\left(\expect{{\rm e} ^{\kappa Y}}\right)^n\pr{N=n}\right)\frac{1}{\kappa m}\nonumber\\
&=&\left(\pr{\tilde{\tau}>\xi}\expect{{\rm e}^{\kappa Y}}+\frac{\pr{\tilde{\tau}\leq\xi}\pr{\tau>\xi}(\expect{{\rm e}^{\kappa Y}})^2}{1-\pr{\tau\leq\xi}\expect{{\rm e}^{\kappa Y}}}\right)\frac{1}{\kappa m}.
\end{eqnarray}
under assumption that
$$m= \varphi^\prime_k(\kappa) <\infty.$$

\begin{Remark}\rm
The net profit condition (NPC) results from \eqref{reprx1}.
By \eqref{numer}, the NPC holds when $\mathbb{E}[X_1]<0$ that is when
\begin{equation*}
\mathbb{E}[X_1]=\mathbb{P}(\tilde{\tau}\leq\xi)\left[\mathbb{E}[Y]-\mathbb{E}[\tau^{>\xi}]+\mathbb{E}[N-1](\mathbb{E}[Y]-\mathbb{E}[\tau^{\leq\xi}])\right] + (\mathbb{E}[Y]-\mathbb{E}[\tilde{\tau}])<0.
\end{equation*}
\end{Remark}

\begin{Example}\label{ex7}\rm
A special example of exponentially distributed $\tau\sim {\rm Exp}(\lambda_1)$, $\tilde{\tau}\sim {\rm Exp}(\lambda_2)$ and $Y\sim {\rm Exp}(\beta)$ would lead to
\begin{eqnarray}
\varphi(\theta)&=&\frac{\lambda_1\lambda_2\left({\rm e}^{-\lambda_1\xi}-{\rm e}^{-\lambda_2\xi}\right)\frac{\hat{B}^{2}(\theta){\rm e}^{-\theta\xi}}{(\lambda_1+\theta)(\lambda_2+\theta)}+\frac{\lambda_2}{\lambda_2+\theta}\hat{B}(\theta){\rm e}^{-(\lambda_2+\theta)\xi}}{1-\frac{\lambda_1}{\lambda_1+\theta}\left(1-{\rm e}^{-(\lambda_1+\theta)\xi}\right)\hat{B}(\theta)}\label{eqn:mgfx2}\\
&=&\left[\lambda_1\lambda_2\left({\rm e}^{-\lambda_1\xi}-{\rm e}^{-\lambda_2\xi}\right)\frac{\beta^2{\rm e}^{-\xi \theta}}{{(\beta-\theta)}^2(\lambda_1+\theta)(\lambda_2+\theta)}+\frac{\lambda_2}{\lambda_2+\theta}\frac{\beta}{\beta-\theta}{\rm e}^{-(\lambda_2+\theta)\xi}\right]\div\nonumber\\
&&\left[1-\frac{\lambda_1}{\lambda_1+\theta}\left(1-{\rm e}^{-(\lambda_1+\theta)\xi}\right)\frac{\beta}{\beta-\theta}\right]
\end{eqnarray}
and
\begin{equation*}
\tilde{K} = \frac{\beta}{\beta-\kappa}\cdot\frac{\beta {\rm e}^{-\lambda_1 \xi}-\kappa{\rm e}^{-\lambda_2\xi}}{\beta{\rm e}^{-\lambda_1\xi}-\kappa}.
\end{equation*}
This gives that
\[
\lim_{x\to\infty}\psi(x){{\rm e}}^{\kappa x} \leq \frac{\beta}{\beta-\kappa}\cdot\frac{\beta {\rm e}^{-\lambda_1 \xi}-\kappa{\rm e}^{-\lambda_2\xi}}{\beta{\rm e}^{-\lambda_1\xi}-\kappa}.
\]

Moreover, since
\begin{eqnarray*}
\mathbb{E}[\tau^{\leq\xi}]&=&\mathbb{E}[\tau|\tau\leq \xi]=\frac{\frac{1}{\lambda_1}-\left(\xi+\frac{1}{\lambda_1}\right){\rm e}^{-\lambda_1\xi}}{1-{\rm e}^{-\lambda_1\xi}};\\
\mathbb{E}[\tau^{>\xi}]&=&\mathbb{E}[\tau|\tau> \xi]=\xi+\frac{1}{\lambda_1},
\end{eqnarray*}
the NPC condition is equivalent to
\begin{equation}
\left(\frac{1}{\beta}-\frac{1}{\lambda_1}\right)(1-{\rm e}^{-\lambda_2\xi})+\left(\frac{1}{\beta}-\frac{1}{\lambda_2}\right){\rm e}^{\lambda_1\xi}<0.
\label{eqn:npc}
\end{equation}

Furthermore, as a connection with Section \ref{sec: numerical}, it is worth mentioning here that the above identity should coincide with
\begin{equation}
\left(\frac{1}{\beta}-\frac{1}{\lambda_1}\right)\pi_1 + \left(\frac{1}{\beta}-\frac{1}{\lambda_2}\right)\pi_2<0,
\end{equation}
where
\begin{eqnarray}
\pi_1& =& \frac{1-{\rm e}^{-\lambda_2\xi}}{1-{\rm e}^{-\lambda_2\xi}+{\rm e}^{-\lambda_1 \xi}},\label{pi1}\\
\pi_2& =& \frac{{\rm e}^{-\lambda_1\xi}}{1-{\rm e}^{-\lambda_1\xi}-{\rm e}^{-\lambda_2\xi}},\label{pi2}
\end{eqnarray}
denote the steady state distribution $(\pi_1, \pi_2)$ in the Markovian environment of $\tau$ and $\tilde{\tau}$, which is precisely defined in Section \ref{sec:map}.
That is to say, when the process becomes stationary, the probability to have an inter-arrival time less or equal to $\xi$ (State 1) would be $\pi_1$ while that for it being larger than $\xi$ (State 2) is represented by $\pi_2=1-\pi_1$. The graph depicted in Figure \ref{Fig:steady} below shows an example of this distribution. It could be seen that the probability for State 1 in our case is monotonically increasing with $\xi$. The blue line represents the ratio of probabilities between State 1 and State 2 thus having the same monotonicity as the green line. This will be analysed further via simulation.
\begin{figure}[!htb]
\centering
\scalebox{0.3}[0.3]{\includegraphics{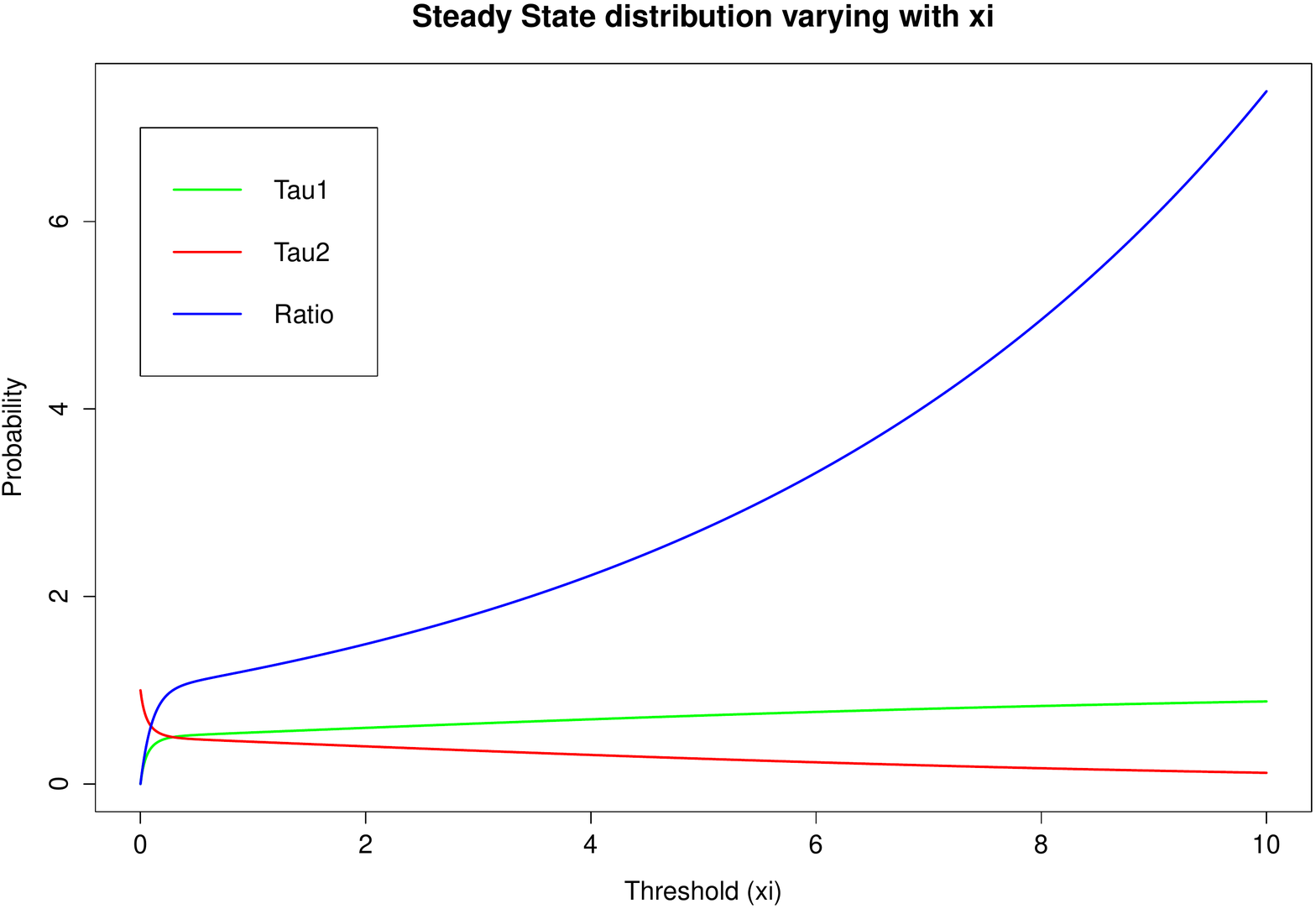}}
\caption{Steady State distribution when $\lambda_1=0.2, \lambda_2=10$}
\label{Fig:steady}
\end{figure}
\end{Example}

\section{Numerical Results}\label{sec: numerical}
In this section, we explain several methods to simulate the ruin probabilities for our model. Initially, we tried the crude Monte Carlo simulation, but as in the classical case, several issues remain including determining a maximum time range so that it approximates an infinite time ruin probability. Then we employed the Importance Sampling technique, which allows us to simulate ruin probabilities under a new measure where ruin happens for sure. However, this does not give us the ruin probability as defined in our original problem. Hence, after a deeper analysis, the construction of a Markov additive further assists in developing a more sophisticated importance sampling technique for our model when the inter claim times as well as claims are exponentially distributed. Using this method, ruin probabilities could be simulated via a closed form formulae. At the end of this section, we present a case study using the crude Monte Carlo simulation where ruin probabilities for our model are compared with the ones under a classical setting, aiming at answering the question which originated this work. We also investigate the influence of two different claim distributions on simulated ruin probabilities.

\subsection{Importance Sampling and Change of Measure}
One cause of the drawback of using the crude Monte Carlo simulation is that ruin probabilities are inefficient, i.e., ruin probability tends to zero very quickly, when the initial reserve $u$ is large. This has been explained by the Cram\'{e}r theorem that asymptotically ruin probability has an exponentially decay with respect to $u$. The other reason of not simply adopting a crude Monte Carlo simulation is that we are anyway trying to simulate an infinite time ruin probability under a finite time horizon. In order to overcome this effect, the importance sampling technique has been brought in. The key idea behind is to find an equivalent probability measure under which the process has a probability of ruin equal to 1.\\

Let us start from something trivial. For the moment, we only consider the "ruin probability" when the time between regenerative epochs is ignored. In other words, we now look at our process from a macro perspective and it is renewal at each regenerative time epoch, so we omit the situations where ruin happens within these intervals. We refer to it as the "macro" process which coincides with a classical risk process and its corresponding ruin probability as the "macro" ruin probability in the sequel. We can then define the macro ruin time as
\begin{equation}
T^{*}(x)= \inf\:\{T_i\geq 0:\,U(T_i)<0, \,i = 1, \ldots \mid U(0)=x\}.
\label{eqn:mruintime}
\end{equation}
Consequently, the macro ruin probability denoted by $\psi^{*}(x) =\mathbb{P}(T^{*}(x)<\infty \mid U(0)=x)$ should be smaller or equal than the ruin probability associated with our actual risk process $\psi(u)$. But for illustration purposes, it is worth covering the nature of change of measure under the framework of this macro process first before we dig into more complex scenarios.
\begin{Theorem}
Assume that there exists a $\kappa$ such that $\varphi(\kappa)=1$.
Consider a new measure $\mathbb{Q}$ such that:
\begin{eqnarray*}
\mathbb{Q}(Y\in dy)&=&\frac{\mathbb{P}(Y\in dy){\rm e}^{\kappa Y}}{\mathbb{E}[{\rm e}^{\kappa Y}]},\\
\mathbb{Q}(\tau^{\leq\xi}\in dx)&=&\frac{\mathbb{P}(\tau\in dx){\rm e}^{-\kappa x}}{\int^{\xi}_{0} {\rm e}^{-\kappa x}\mathbb{P}(\tau\in dx)}, \, x\in(0, \xi],\\
\mathbb{Q}(\tau^{>\xi}\in dx)&=&\frac{\mathbb{P}(\tau\in dx){\rm e}^{-\kappa x}}{\int^{\infty}_{\xi} {\rm e}^{-\kappa x}\mathbb{P}(\tau\in dx)}, \, x\in(\xi, \infty)
\end{eqnarray*}
with $\tilde{\tau}^{\leq\xi}$ and $\tilde{\tau}^{>\xi}$ defined in a similar way.
Then we could establish the same relation as in the classical case for the m.g.f. of $X_1$,
\begin{equation}
\varphi_{\mathbb{Q}}(\theta)=\varphi(\theta+\kappa)/\varphi(\kappa)=\varphi(\theta+\kappa).
\label{phirelation}
\end{equation}
\label{thm:chgmeasure}
\end{Theorem}
\begin{proof}
Rewriting the equation \eqref{eqn:mgfx2} we derive:
\begin{eqnarray}
\varphi(\theta+\kappa)&=&\mathbb{E}[{\rm e}^{(\theta+\kappa)Y}]\mathbb{E}[{\rm e}^{-(\theta+\kappa)\tilde{\tau}}, \tilde{\tau}>\xi]+\mathbb{E}[{\rm e}^{(\theta+\kappa)Y}]\mathbb{E}[{\rm e}^{-(\theta+\kappa)\tilde{\tau}}, \tilde{\tau}\leq\xi]\nonumber\\
&&\cdot\mathbb{E}[{\rm e}^{(\theta+\kappa)Y}]\mathbb{E}[{\rm e}^{-(\theta+\kappa)\tau}, \tau>\xi]\sum_{k=1}^{\infty}\left((1-p)\mathbb{E}[{\rm e}^{(\theta+\kappa)Y}]\mathbb{E}[{\rm e}^{-(\theta+\kappa)\tau}, \tau\leq\xi]\right)^{k-1}
\label{phi+}
\end{eqnarray}

First, we note that
\begin{equation}
\mathbb{E}[{\rm e}^{(\theta+\kappa)Y}]=\int{\rm e}^{(\theta+\kappa)Y}\mathbb{P}(Y\in dy)=\mathbb{E}[{\rm e}^{\kappa Y}]\int {\rm e}^{\theta Y}\mathbb{Q}(Y\in dy)=\mathbb{E}[{\rm e}^{\kappa Y}]\mathbb{E}_{\mathbb{Q}}[{\rm e}^{\theta Y}].\label{sepy}
\end{equation}
Thus for $\tau^{\leq\xi}, \tau^{>\xi}$,
\begin{eqnarray}
\mathbb{E}[{\rm e}^{-(\theta+\kappa)\tau}, \tau>\xi]&=&\mathbb{E}[{\rm e}^{-\kappa \tau},\tau>\xi]\mathbb{E}_{\mathbb{Q}}[{\rm e}^{-\theta\tau^{>\xi}}],\\
\mathbb{E}[{\rm e}^{-(\theta+\kappa)\tau}, \tau\leq\xi]&=&\mathbb{E}[{\rm e}^{-\kappa \tau},\tau\leq\xi]\mathbb{E}_{\mathbb{Q}}[{\rm e}^{-\theta\tau^{\leq\xi}}].
\label{septau}
\end{eqnarray}
Note that $\tilde{\tau}^{\leq\xi}, \tilde{\tau}^{>\xi}$ have the same form. Then the equation (\ref{phi+}) could be modified into:
\begin{eqnarray*}
\varphi(\theta+\kappa)&=&\mathbb{E}[{\rm e}^{\kappa Y}]\mathbb{E}[{\rm e}^{-\kappa \tilde{\tau}},\tilde{\tau}>\xi]\cdot\left[\mathbb{E}_{\mathbb{Q}}[{\rm e}^{\theta Y}]\mathbb{E}_{\mathbb{Q}}[{\rm e}^{-\theta\tilde{\tau}^{>\xi}}]\right]\\
&+&\mathbb{E}[{\rm e}^{\kappa Y}]\mathbb{E}[{\rm e}^{-\kappa \tilde{\tau}},\tilde{\tau}\leq\xi]\cdot\left[\mathbb{E}_{\mathbb{Q}}[{\rm e}^{\theta Y}]\mathbb{E}_{\mathbb{Q}}[{\rm e}^{-\theta\tilde{\tau}^{\leq\xi}}]\right]\\
&\cdot&\mathbb{E}[{\rm e}^{\kappa Y}]\mathbb{E}[{\rm e}^{-\kappa \tau},\tau>\xi]\cdot\left[\mathbb{E}_{\mathbb{Q}}[{\rm e}^{\theta Y}]\mathbb{E}_{\mathbb{Q}}[{\rm e}^{-\theta\tau^{>\xi}}]\right]\\
&\cdot&\sum_{k=1}^{\infty}\left(\mathbb{E}[{\rm e}^{\kappa Y}]\mathbb{E}[{\rm e}^{-\kappa \tau},\tau\leq\xi]\right)^{k-1}\cdot\left[\mathbb{E}_{\mathbb{Q}}[{\rm e}^{\theta Y}]\mathbb{E}_{\mathbb{Q}}[{\rm e}^{-\theta\tau^{\leq\xi}}]\right]^{k-1}.
\end{eqnarray*}

Now let,
 \begin{eqnarray*}
\tilde{p}_{\kappa}&=&\mathbb{E}[{\rm e}^{\kappa Y}]\mathbb{E}[{\rm e}^{-\kappa \tilde{\tau}},\tilde{\tau}>\xi], \quad\tilde{q}_{\kappa}=1-\tilde{p}_{\kappa},\\
p_\kappa&=&(\mathbb{E}[{\rm e}^{\kappa Y}])^2\mathbb{E}[{\rm e}^{-\kappa \tilde{\tau}},\tilde{\tau}\leq\xi]\mathbb{E}[{\rm e}^{-\kappa \tau},\tau>\xi],\quad q_\kappa=1-p_\kappa,
\end{eqnarray*}
Then,
\begin{eqnarray*}
\varphi(\theta+\kappa)&=&\tilde{p}_{\kappa}\cdot\left[\mathbb{E}_{\mathbb{Q}}[{\rm e}^{\theta Y}]\mathbb{E}_{\mathbb{Q}}[{\rm e}^{-\theta\tilde{\tau}^{>\xi}}]\right]+p_\kappa\tilde{q}_\kappa\left[(\mathbb{E}_{\mathbb{Q}}[{\rm e}^{\theta Y}])^2\mathbb{E}_{\mathbb{Q}}[{\rm e}^{-\theta\tilde{\tau}^{\leq\xi}}]\mathbb{E}_{\mathbb{Q}}[{\rm e}^{-\theta\tau^{>\xi}}]\right]\\
&&\cdot\sum_{k=1}^{\infty}\left(1-p_\kappa\right)^{k-1}\cdot\left[\mathbb{E}_{\mathbb{Q}}[{\rm e}^{\theta Y}]\mathbb{E}_{\mathbb{Q}}[{\rm e}^{-\theta\tau^{\leq\xi}}]\right]^{k-1}\\
&=&\varphi_{\mathbb{Q}}(\theta).
\end{eqnarray*}
\end{proof}

To analyse \eqref{phirelation} further, $\varphi_\mathbb{Q}(\theta)$ can be considered as if the function $\varphi(\theta)$ shifted to the left by $\kappa$. We know that the net profit condition for the macro process requires $\expect{X_1}<0$, i.e., $\varphi'(0)<0$. Additionally, \eqref{phitheta} should have a positive root $\kappa$ if the tail of the claim cost distribution is exponentially bounded. That is to say, $\varphi'(0)>0$ would result in a positive drift of the macro claim surplus process and then cause a macro ruin to happen for certain. The new m.g.f. $\varphi_{\mathbb{Q}}(\theta)=\varphi(\theta+\kappa)$ makes this true. Hence
we can write for a macro ruin probability as
\begin{equation*}
\psi^{*}(x)=\mathbb{E}[{\bf 1}_{T^{*}(x)<\infty}]=\mathbb{E}_{\mathbb{Q}}[{\rm e}^{-\kappa S(T^{*}(x))+T^{*}(x)\ln\varphi(\kappa) }{\bf 1}_{T^{*}(x)<\infty}]
\end{equation*}
with $\mathbb{E}_{\mathbb{Q}}[{\bf 1}_{T^{*}(x)<\infty}] = 1$.
For a strict and detailed proof please refer to \cite{RP} (Chapter IV. Theorem 4.3). Moreover, from
\eqref{phirelation} it follows that
\begin{equation}\label{x1new}
\mathbb{P}_{\mathbb{Q}}(X_1\in dy) = \frac{\mathbb{P}(X_1\in dy) {\rm e}^{\kappa y}}{\int_R \mathbb{P}(X_1\in dz) {\rm e}^{\kappa z} dz}
\end{equation}
and $\mathbb{Q}$ is absolutely continuous with respect of $\mathbb{P}$ (up to time $n$)
with a likelihood ratio:
\begin{equation}
L_n =
 {\rm e}^{n\ln\varphi(\kappa)-\kappa \sum_{i = 1}^{n} X_{i}}.
\end{equation}
 Define a new stopping time  $N^*(x) =\inf \{n\geq 0;\;S_n>x\}$. Note that the event $\{N^*(x)<\infty\}$ is equivalent to $\{T^*(x)<\infty\}$.
 From the Optional Stopping Theorem it follows that
 for any set $G\subseteq\{N^*(x) < \infty\}$ we have
 \begin{equation*}
 \mathbb{P}\{G\} = \mathbb{E}_{\mathbb{Q}}\left[\frac{1}{L_{N^*(x)}}; \,G\right].
 \end{equation*}
See \cite[Chapter III. Theorem 1.3]{RP} for more details.\\

This means that we could simulate macro ruin probabilities under the new measure $\mathbb{Q}$ where the ruin happens with probability $1$.
We do it using new law of $X_1$ given in \eqref{x1new} and to each ruin even we add weight $\frac{1}{L_{N^*(x)}}$ where $N^*(x)$ is the observed macro ruin time.
Summing all events with weights produces the ruin probability $\psi^*(x)$.
In this way we can avoid infinite time simulations. For the case when everything is exponentially distributed as it was considered in Example \ref{ex7}
we have that under $\mathbb{Q}$ the simulation should be made according to new parameters:
\begin{eqnarray*}
&Y^{(\kappa)}\sim{\rm Exp}(\beta-\kappa), \\
&{\tilde\tau}^{>\xi}\sim{\rm Exp}(\lambda_2+\kappa)\; \text{on} \;(\xi,\infty),\\
&\tau^{>\xi}\sim{\rm Exp}(\lambda_1+\kappa)\; \text{on} \;(\xi,\infty),\\
&\tau^{\leq\xi}\sim{\rm Exp}(\lambda_1+\kappa)\; \text{on} \;(0,\xi],\\
&\tilde{\tau}^{\leq\xi}\sim{\rm Exp}(\lambda_2+\kappa)\; \text{on} \;(0,\xi].
\end{eqnarray*}

In this case $X_1$ has the law of
\begin{equation}
\mathbf{I}_{\{Z>\tilde{p}_\kappa\}}(Y_0-\tilde{\tau}_0^{>\xi})+\mathbf{I}_{\{Z\leq\tilde{p}_\kappa\}}\left(\sum_{i=1}^{N-1}(Y_i-\tilde{\tau}_i^{\leq\xi})+\left(Y_{N}-\tilde{\tau}^{>\xi}_{N}\right)+\left(Y_{0}-\tilde{\tau}^{\leq\xi}_{0}\right)\right),
\end{equation}
\label{chgprocess}
where $N\sim {\rm Geo}(p_\kappa)$ and $Z\sim U(0,1)$.

\subsection{Embedded Markov additive process}
\label{sec:map}
To get more precise simulation results avoiding the macro ruin probability giving lower estimate only, we have to understand the structure of our process better.
To implement this, we will use the theory of discrete-time Markov Additive Processes. For simplicity, we assume everything to be exponential distributed with  $\tau\sim {\rm Exp}(\lambda_1)$, $\tilde{\tau}\sim {\rm Exp}(\lambda_2)$ and $Y\sim {\rm Exp}{(\beta)}$, respectively.\\

Recall our process described by \eqref{eqn:XM}, note that ruin happens only at claim arrivals $\sigma_k=\sum_{i=1}^k\tau_i$ and $\sigma_0=0$. From time $\sigma_k$ to $\sigma_{k+1}$, the distribution of the increment $S(\sigma_{k+1})-S(\sigma_k)$ is only dependents on the relation between $\tau_k$ and $\xi$. Hence, we could transfer the original model $S_n$
given in \eqref{numer}
into a new one $(S_n,J_n)$ ($n\geq 0$)
by adding a Markov state process $\{J_n\}_{n\geq0}$ defined on $E=\{1, 2\}$. The index $i \in E$ represents the occupying state of $\{J_k\}$ at time $\sigma_k$.
For instance, state $1$ describes a status where the current inter-arrival time is less or equal than $\xi$ while state $2$ refers to the opposite situation. For convenience, we construct $\tau_0$ based on the choice of $J_0$: $J_0=1$ implies $\tau_0\leq\xi$ and $\tau_0>\xi$ otherwise. Note that the two state Markov chain $\{J_n\}$ has a transition probability matrix as follows with the $ij^{th}$ element being $p_{ij},\; i, j\in E$.
\begin{eqnarray*}
\bf{P}= \left[\begin{array}{ccc}
q & p\\
\tilde{q} & \tilde{p} \end{array}\right],
\end{eqnarray*}
where $p=\pr{\tau> \xi}, q=1-p=\pr{\tau\leq \xi}$ and $\tilde{p}=\pr{\tilde{\tau}> \xi}, \,\tilde{q}=1-\tilde{p}=\pr{\tilde{\tau}\leq \xi}$. We also define a new process $\{S_n\}_{n\geq0}$ whose increment $\Delta S_{n+1}=S_{n+1}-S_{n}$ is governed by $\{J_n\}$. More specifically, two scenarios could be analysed to explain this process. Given $n=0,1,\ldots$, scenario 1 is when $J_n=1$,  i.e., $\tau_n\leq \xi$ and $\tau_{n+1}\overset{d}{=}\tau$. Then comparing $\tau$ with $\xi$, there is a chance $q$ of obtaining $J_{n+1}=1$ given $\tau\leq\xi$, and $p$ having $J_{n+1}=2$ given $\tau>\xi$, with the corresponding increment being $\Delta S_{n+1}\overset{d}{=} Y-\tau^{\leq \xi}$ and $\Delta S_{n+1}\overset{d}{=} Y-\tau^{> \xi}$, respectively. On the contrary, scenario 2 represents the situation where the current state is $J_n =2$, i.e., $\tau_n> \xi$ and $\tau_{n+1}\overset{d}{=}\tilde{\tau}$. Thus, all the variables above are presented in the same way only with a tilde sign added on $\tau$, $p$ and $q$.\\

$Z_n=(S_n, J_n)$ is a discrete time bivariate Markov process also referred to as a discrete-time Markov additive process (MAP).
The moment of ruin is the first passage time of $S_n$ over level $x>0$, defined by
$$T^{(i)}(x)=\inf\{n\in\mathbb{N}:\;S_n>u |Z_0=(0,i)\},\qquad \hbox{ for }i=1,2.$$
Without loss of generality, assume that $\sigma_{{T^{(2)}(x)}} = T(x)$. Then the event $\{T^{(2)}(x)<\infty\}$ is equivalent to $\{T(x)<\infty\}$.
This implies that
$$\psi(x)=\mathbb{P}(T^{(2)}(x)<\infty).$$

To perform simulation
we will derive now the special representation of the underlying ruin probability using
new change of measure.
We start from identifying
a kernel matrix $F_{ij}(dx)$ with the $ij^{th}$ entry given by $F_{ij}(dx) = \mathbb{P}_{i}(J_1 = j, \,\Delta S_1 \in dx)$.
Here $\mathbb{P}_i$ and $\mathbb{E}_i$ denotes the probability measure conditional on the event $\{J_0=i\}$
and its corresponding expectation, respectively. Then for $\theta>0$, a m.g.f of the measure $F_{ij}(dx)$ is $\hat{F}_{ij}[\theta]=\mathbb{E}_i [{\rm e}^{\theta \Delta S_1}; J_1= j]$
with
\begin{eqnarray*}
\hat{\bf{F}}[\theta] = \left[\begin{array}{ccc}
\mathbb{E}({ \rm e}^{\theta Y} {\rm e}^{-\theta \tau}; \tau \leq\xi ) & \mathbb{E}({\rm e}^{\theta Y}{\rm e}^{-\theta  \tau}; \tau > \xi) \\
\mathbb{E}({\rm e}^{\theta Y} {\rm e}^{-\theta \tilde{\tau}}; \tilde{\tau} \leq\xi ) & \mathbb{E}({\rm e}^{\theta Y}{\rm e}^{-\theta \tilde{\tau}}; \tilde{\tau} >\xi)
\end{array}\right].
\end{eqnarray*}
Additionally, based on the additive structure of the process $Z_n$ for $\hat{F}_{n,ij}[\theta] = \mathbb{E}_i[{\rm e}^{\theta (S_n-S_0)}; J_n= j]$
we have:
$$\hat{{\bf F}}_n[\theta] = (\hat{{\bf F}}[\theta])^n.$$

We will now present few facts that will be used in the main construction.
\begin{Lemma}
We have,
\begin{equation}
\mathbb{E}_{J_{n}}[{\rm e}^{\theta(S_{n+1}-S_n)}v_{J_{n+1}}^{(\theta)}]= \lambda(\theta)v_{J_n}^{(\theta)},
\end{equation}
where $\lambda(\theta)$ is the eigenvalue of $\hat{\bf {F}}[\theta]$ and ${\bf v}^\theta=(v_1^\theta, v^\theta_2)^{T}$ is the corresponding right eigenvector.
\end{Lemma}

\begin{proof}
Note that
\begin{eqnarray*}
\mathbb{E}_{J_n}[{\rm e}^{\theta(S_{n+1}-S_n)}v_{J_{n+1}}^{(\theta)}]={\bf e}_{J_n}^T \hat{\bf {F}}_1[\theta]{\bf v}={\bf e}_{J_n}^T\lambda(\theta){\bf v}=\lambda(\theta)v_{J_n}^{(\theta)},
\end{eqnarray*}
where ${\bf e}_{J_n}$ is a standard basis vector. This completes the proof.
\end{proof}

\begin{Lemma}
The following sequence
\begin{equation}
L_n={\rm e}^{\theta S_n-n\ln \lambda(\theta)} \frac{v_{J_{n}}^{(\theta)}}{v_{J_{0}}^{(\theta)}}
\end{equation}
is a discrete-time martingale.
\label{lem:martingale}
\end{Lemma}

\begin{proof}
Let $M_n=L_nv_{J_{0}}^{(\theta)}.$ Then,
\begin{eqnarray*}
\mathbb{E}[M_{n+1}|\mathcal{F}_n]&=&\mathbb{E}[{\rm e}^{\theta S_{n+1}-(n+1)\ln \lambda(\theta)} v_{J_{n+1}}^{(\theta)}|\mathcal{F}_n]\\
&=&\mathbb{E}[{\rm e}^{\theta(S_{n+1}-S_n)}v^{(\theta)}_{J_{n+1}}|\mathcal{F}_n]{\rm e}^{\theta S_n-(n+1)\ln \lambda(\theta)}\\
&=&\mathbb{E}_{J_{n}}[{\rm e}^{\theta(S_{n+1}-S_n)}v_{J_{n+1}}^{(\theta)}]{\rm e}^{\theta S_n-(n+1)\ln \lambda(\theta)}\\
&=&\lambda(\theta)v_{J_n}^{(\theta)}{\rm e}^{\theta S_n-(n+1)\ln \lambda(\theta)}\\
&=&M_n,
\end{eqnarray*}
which gives the assertion of the lemma.
\end{proof}

Define now a new conditional probability measure $\mathbb{Q}^{(\theta)}_{i}(dx)=\mathbb{Q}^{(\theta)}(dx|J_0=i)$ using Randon-Nikodym derivative as follows:
$${\frac{d\mathbb{Q}^{(\theta)}_i}{d\mathbb{P}_i}}=L_n.$$

\begin{Lemma}
Under the new measure $\mathbb{Q}$ process $\{Z_n^{(\theta)}\}_{n\in\mathbb{N}}$ is again MAP specified by the Laplace transform of its kernel in the following way:
\begin{eqnarray}
\hat{{\bf F}}^{(\theta)}[\gamma]={\rm e}^{-\ln\lambda(\theta)}\left({\bf v}_{diag}^{(\theta)}\right)^{-1}\hat{\bf F}[\theta+\gamma]{\bf v}^{(\theta)}_{diag},
\end{eqnarray}
where ${\bf v}^{(\theta)}_{diag}$ is a diagonal matrix with ${\bf v}^{(\theta)}$ on the diagonal.
\end{Lemma}

\begin{proof}
Note that the kernel $F^{(\theta)}_{ij}(dx)$ of $Z_n$ can be written as:
\begin{eqnarray*}
F^{(\theta)}_{ij}(dx) &=& \mathbb{Q}^{(\theta)}_i(S_1\in dx, J_1=j) =\mathbb{E}^{\mathbb{Q}^{(\theta)}}[{\bf 1}_{\{S_1\in dx, J_1=j\}}|J_0=i]= \mathbb{E}_i[L_1{\bf 1}_{\{S_1\in dx, J_1 = j\}}]\\
&=&
{\rm e}^{\theta x-\ln \lambda(\theta)} \frac{v_{j}^{(\theta)}}{v_{i}^{(\theta)}}F_{ij}(dx).
\end{eqnarray*}
This shows that the new measure is exponentially proportional to the old one, which ensures that $F^{(\theta)}_{ij}$ is absolutely continuous with respect to $F_{ij}$. Further transferring it into the matrix m.g.f. form yields the desired result.
\end{proof}

\begin{Corollary}\label{newp}
Under the new measure $\mathbb{Q}^{(\theta)}$, the MAP $\{Z_n^{(\theta)}\}_{n\in\mathbb{N}}$ consists of a Markov state process $\{J^{(\theta)}_n\}_{n\in\mathbb{N}}$ which has a transition probability matrix
\begin{eqnarray}
{\bf P}^{(\theta)} = \left[\begin{array}{ccc}
q_{\theta} & p_{\theta}\\
\tilde{q}_{\theta} & \tilde{p}_{\theta} \end{array}\right],
\label{eqn:chgtransition}
\end{eqnarray}
where
\begin{eqnarray*}
\tilde p _{\theta} = \frac{\beta\lambda_2}{(\beta-\theta)(\lambda_2+\theta)} e^{-(\lambda_2+\theta)\xi},  &&\;\tilde q_{\theta} = 1-\tilde p_{\theta},\\
q_{\theta} = \frac{\beta\lambda_1}{(\beta-\theta)(\lambda_1+\theta)}(1- e^{-(\lambda_1+\theta)\xi}),   &&\; p_{\theta} =  1- q_{\theta},
\end{eqnarray*}
and an additive component $\{S^{(\theta)}_n\}_{n\in\mathbb{N}}$ with random variables $Y,\tau^{>\xi},\tau^{<\xi},\tilde{\tau}^{>\xi},\tilde{\tau}^{<\xi}$
with laws given in Theorem \ref{thm:chgmeasure} where $\theta$ should be chosen everywhere instead of $\kappa$.
\end{Corollary}

In fact, when $\theta=\kappa$, $\mathbb{Q}^{(\theta)}$ coincides with $\mathbb{Q}$ defined by Theorem \ref{thm:chgmeasure}. Recall $T^{*}(x)$ from \eqref{eqn:mruintime} and $\psi(x)\geq \psi^{*}(x)$. Since $\sigma_{T^{(2)}(x)}\leq T^{*}(x)$, then $\mathbb{Q}(T^{*}(x)<\infty)=1$ implies $\mathbb{Q}^{(\kappa)}(T^{(2)}(x)<\infty)=1$.
Now the main representation used in simulations follows straightforward from above lemmas and Optional Stopping Theorem
as it was already done in the previous section and it is given in the next theorem.

\begin{Theorem}
The ruin probability for the underlying process \eqref{eqn:XM} equals:
\begin{equation}
\psi(x) = v^{(\kappa)}_{2} e^{-\kappa x} \mathbb{E}^{(\kappa)}_{2}\left[\frac{e^{-\kappa \varepsilon(T^{(2)}(x))}}{v^{(\kappa)}_{J_{T^{(2)}(x)}}}\right],
\label{eqn:mapruin}
\end{equation}
where $\varepsilon(T^{(2)}(x))=S^{(\kappa)}_{T^{(2)}(x))}-u$ denotes the overshoot at the time of ruin $T^{(2)}(x)$.
\end{Theorem}

Now we will simulate ruin events using new parameters of the model identified in Lemma \ref{newp}. We start from state $2$ of $J_0$.
We will run our risk process until ruin event. With each ruin event we will associate its
weight $v^{(\kappa)}_{2} e^{-\kappa x} \frac{e^{-\kappa \varepsilon(T^{(2)}(x))}}{v^{(\kappa)}_{J_{T^{(2)}(x)}}}$.
Summing and averaging all weights gives the estimate of the ruin probability $\psi(x)$.

\begin{Remark}
In addition, it has been discovered that
$\hat{\bf{F}}[\kappa]$ has an eigenvalue equal to 1 and
\begin{eqnarray*}
{\bf v}^{(\kappa)} = \left[\begin{array}{ccc}
\frac{\beta\lambda_1}{(\beta-\kappa)(\lambda_1+\kappa)}-q_{\kappa}\\
p_{\kappa}\end{array}\right]
\end{eqnarray*}
is the corresponding right eigenvector.
\end{Remark}
\begin{proof}
Indeed, let $\lambda$ denote the eigenvalue of $\hat{\bf{F}}[\kappa]$. Thus we can write,
\begin{equation*}
(\mathbb{E}[{\rm e}^{\kappa Y}]\mathbb{E}[{\rm e}^{-\kappa \tau},\tau\leq \xi]-\lambda)(\mathbb{E}[{\rm e}^{\kappa Y}]\mathbb{E}[{\rm e}^{-\kappa \tilde{\tau}} \tilde{\tau}>\xi]-\lambda)=(\mathbb{E}[{\rm e}^{\kappa Y}])^2\mathbb{E}[{\rm e}^{-\kappa \tau},\tau>\xi]\mathbb{E}[{\rm e}^{-\kappa \tilde{\tau}},\tilde{\tau}\leq\xi].
\end{equation*}
Recall \eqref{eqn:mgfx1}, clearly $\lambda=1$ is a solution to the above equation. That directly leads to $\hat{\bf{F}} {\bf v}={\bf v}$ and one can obtain
$$\frac{v_1}{v_2}=\frac{\mathbb{E}[{\rm e}^{-\kappa \tau},\tau>\xi]}{1-\mathbb{E}[{\rm e}^{-\kappa \tau},\tau\leq\xi]}.$$ Plugging in the parameters completes the proof.
\end{proof}

\begin{Example}\rm
Assume that $Y$, $\tau$ and $\tilde\tau$ have exponential distribution with parameters $\beta =3$, $\lambda_1 = 1$ and $\lambda_2 = 2$, respectively.
The smallest positive real root of Equation \eqref{eqn:mgfx1}  is calculated for $\kappa = 1.1439$ and its corresponding right eigenvector is $\bf{v}^{(k)} = [0.5790,  0.8153]'$.
Then the ruin probability is plotted in Figure 8. It shows an exponential decay as we expected.
\begin{figure}[!htb]
\centering
\scalebox{0.25}[0.25]{\includegraphics{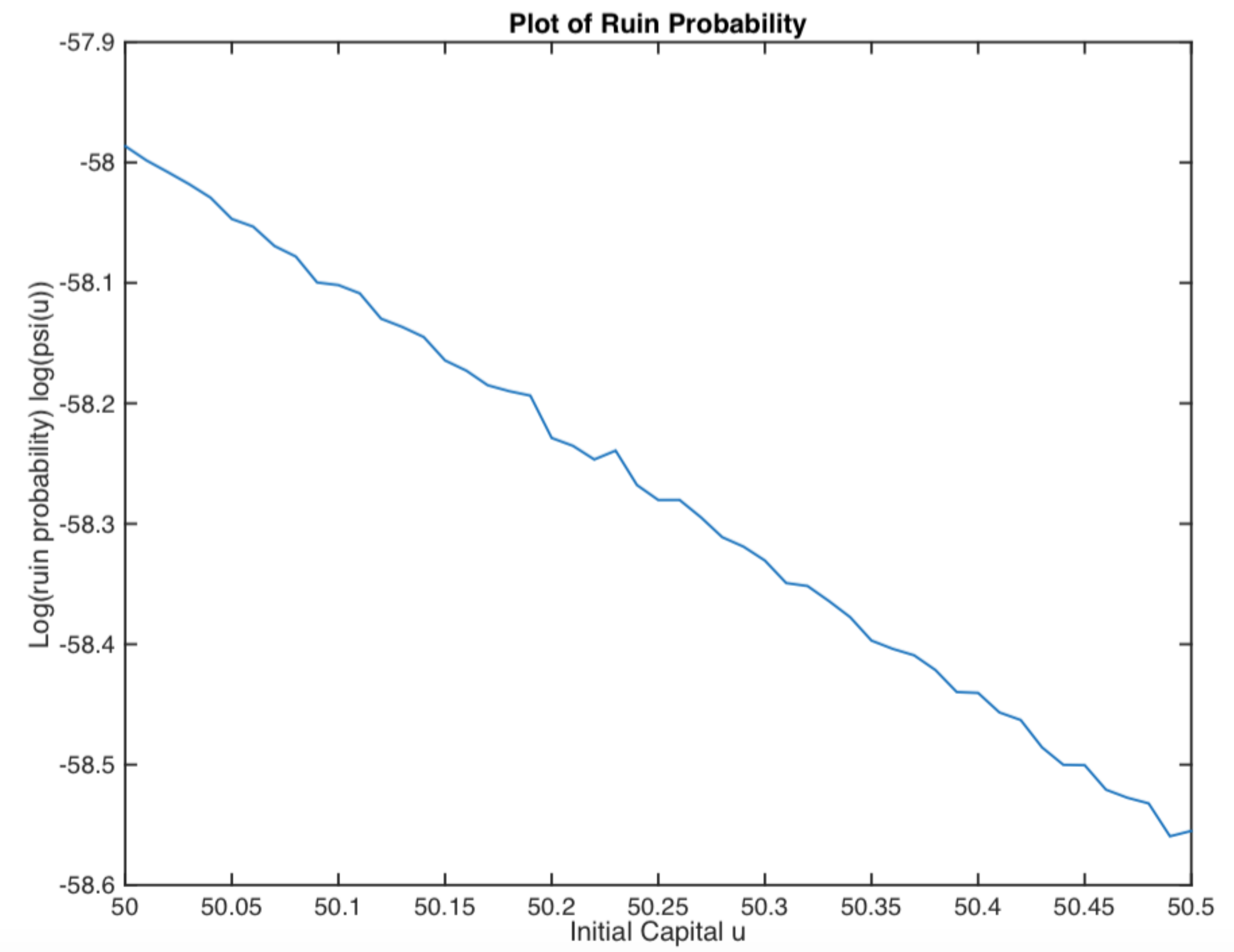}}
\caption{Logarithm of ruin probability}
\label{Fig:rp}
\end{figure}
\end{Example}

\subsection{Case Study}
In this subsection, we show some results via a crude Monte Carlo simulation method. The key idea is to simulate the process according to the model setting and simply counting the number of paths that gets to ruin. Due to the nature of this approach, a 'maximum' time should be set beforehand, which means we are in fact simulating a finite time ruin probability. However, the drawback of it may be ignored for now as long as we are not getting a lot of zeros.\\

Our task is to compare the simulated results with a classical analytical ruin function when exponential claims are considered.
\begin{equation*}
\psi_C(u) = \frac{\lambda}{\beta}{\rm e}^{-(\beta-\lambda)u}.
\end{equation*}
Just to prepare for later explanation, a system of integral equations for our model could actually be written,
\begin{eqnarray}
\psi_1(x)&=&\int^{\xi}_0 f_1(t)g_1(x+t)dt+\int^{\infty}_\xi f_1(t)g_2(x+t)dt,\label{IE1}\\
\psi_2(x)&=&\int^{\xi}_0 f_2(t)g_1(x+t)dt+\int^{\infty}_\xi f_2(t)g_2(x+t)dt,\label{IE2}
\end{eqnarray}
where $\psi_1(x), \psi_2(x)$ correspond to the ruin probabilities with the first inter-arrival time being $\tau$ and $\tilde{\tau}$ respectively, and
\begin{equation*}
g_i(x)=\int^{x}_0 \psi_i(x-y)b(y)dy+\int^{\infty}_x b(y)dy,\quad i=1,2
\end{equation*}
with $b(y)$ being the density function of the claim sizes.
Hence, for the simplest case of exponentially distributed claim costs, we plotted both the classic ruin probabilities and our simulated ones on the same graph as shown below (see Figure \ref{Fig:classic}).\\
\begin{figure}[!htb]
\centering
\scalebox{0.45}[0.45]{\includegraphics{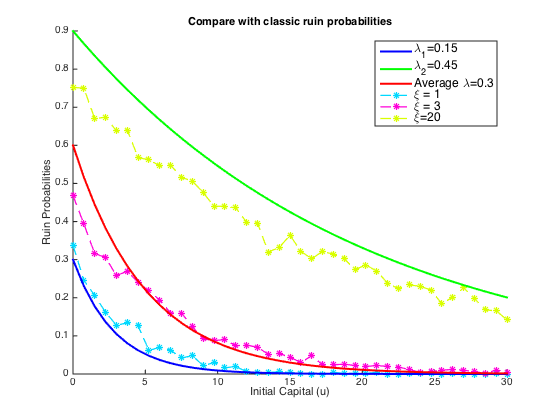}}
\caption{Comparison with classic ruin probabilities with Exponential claims}
\label{Fig:classic}
\end{figure}

It could be concluded that under two given parameters for Poisson intensity, simulated finite ruin probabilities in our model lie between two extreme but have many possibilities in-between. The comparison depends extensively on the value of $\xi$. These results also confirmed Theorem \ref{thm:cramer} that the tail of the ruin function in our case still has an exponential decay and $\xi$ is strongly related to the solution for $\kappa$. In other words, when the dependence is introduced, it is not for sure that ruin probabilities would see an improvement. \\ 

Moving into details, solid lines show classical ruin probabilities (infinite-time) as a function of initial reserve $u$, and each of them denotes an individual choice of Poisson parameters ($\lambda_1=0.15$, $\lambda_2=0.45$) with the middle one being the average of the other two ($\lambda=0.3$). It is clear that the larger the Poisson parameter, the higher is the ruin probability. On the other hand, those dotted lines are simulated results from our risk model with dependence for the same given pair of Poisson parameters $\lambda_1=0.15$ and $\lambda_2=0.45$. The four layers here correspond to four different choices of values for $\xi$, i.e., $\xi=1,\xi=3,\xi=4.44,\xi=20$. If $\xi\rightarrow 0$, the simulated ruin probability (in fact finite-time) tends to a classical case with the lower claim arrival intensities ($\lambda_1$ here), which explains the blue dotted line lying around the dark blue solid line. On the contrary, if $\xi\rightarrow \infty$, simulated ruin probabilities approach the other end. This phenomenon is also theoretically supported by the integral equations \eqref{IE1} and \eqref{IE2} if either of these limits ($\xi\rightarrow 0$ and $\xi\rightarrow\infty$) is taken. This then triggered us to search for a $\xi$ such that the simulated ruin probability coincides with a classical one. Let us see an example here, if $\xi=\frac{\frac{1}{\lambda_1}+\frac{1}{\lambda_2}}{2} = 4.44$ based on the parameters we chose in Figure \ref{Fig:classic}. That implies the choice of our fixed window is the average length of the two kinds of inter-arrival times. However, as can be seen from Figure \ref{Fig:classic}, the dotted line with $\xi=3$ lies closer than the one with $\xi=4.44$ to the red solid line. This suggests that the choice of $\xi$ will influence the simulated ruin probabilities and thus the comparison with a classical one. It is also very likely that there exists a $\xi$ such that our simulated ruin probabilities concur with the classic one.\\

While the first half of the Monte Carlo simulation looked at the influence of $\xi$ on simulated ruin probabilities, the second step is to see the effects of claim sizes. Typical representation of light-tailed and heavy-tailed distributions - Exponential and Pareto - were assumed for claim severities and inter arrival times were switching between two different exponentially distributed random variables with parameters $\lambda_1$ and $\lambda_2$. Two cases were simulated - either $\lambda_1>\lambda_2$ or $\lambda_1<\lambda_2$. It is expected that the effects from claim severity distributions on infinite time ruin probabilities would be tiny as they normally affects more severely in the deficit at ruin. Here, since we simulate finite-time ruin probabilities, we are curious whether the same conclusion can be drawn.\\

\begin{figure}[!htb]
\centering
\begin{subfigure}[b]{7cm}
\scalebox{0.35}[0.35]{\includegraphics{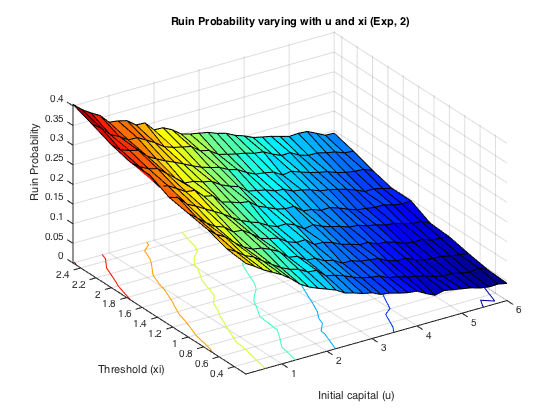}}
\caption{Ruin probabilities when $\lambda_1=0.45, \lambda_2=0.15, \beta = 0.5$}
\label{Fig:rpexp_higher1}
\end{subfigure}

\hspace{1cm}

\begin{subfigure}[b]{7cm}
\centering
\scalebox{0.35}[0.35]{\includegraphics{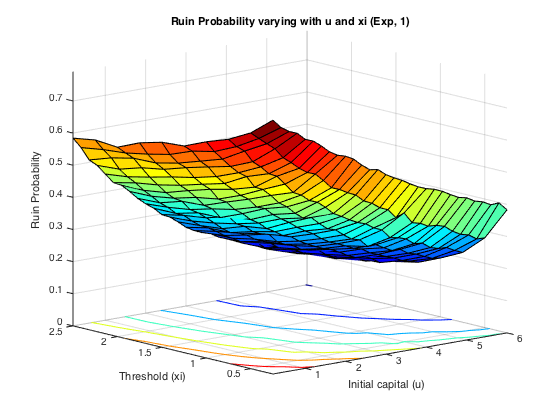}}
\caption{Ruin probabilities when $\lambda_1=0.15, \lambda_2=0.45, \beta = 0.5$}
\label{Fig:rpexp_higher2}
\end{subfigure}
\caption{Examples: Ruin probabilities for Exponential Claims}\label{Fig:rpexp}
\end{figure}

\begin{figure}[!htb]
\centering
\begin{subfigure}[b]{7cm}
\scalebox{0.35}[0.35]{\includegraphics{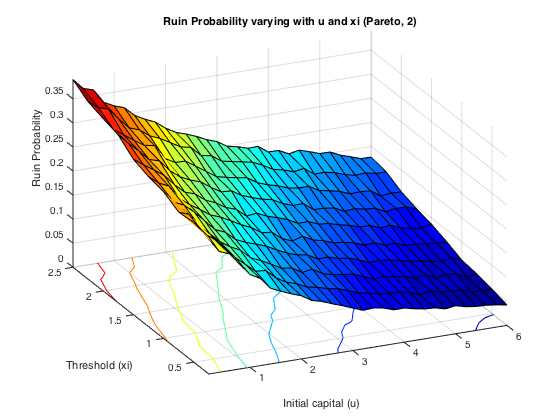}}
\caption{Ruin probabilities when $\lambda_1 = 0.45, \lambda_2 = 0.15, \alpha = 2$}
\label{Fig:rppar_higher1}
\end{subfigure}

\hspace{1cm}

\begin{subfigure}[b]{7cm}
\centering
\scalebox{0.35}[0.35]{\includegraphics{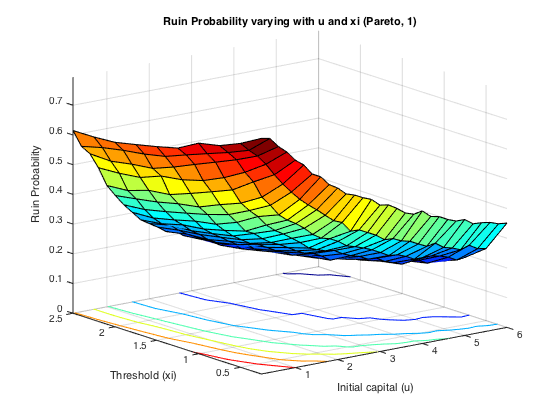}}
\caption{Ruin probabilities when $\lambda_1 = 0.15, \lambda_2 = 0.45, \alpha=2$}
\label{Fig:rppar_higher2}
\end{subfigure}
\caption{Examples: Ruin probabilities for Pareto claims}\label{Fig:rppar}
\end{figure}

\begin{figure}[!htb]
\centering
\scalebox{0.35}[0.35]{\includegraphics{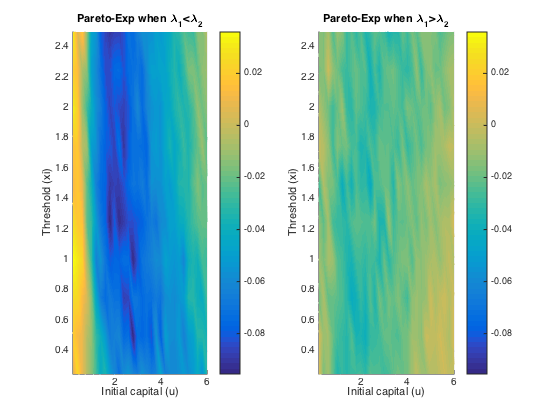}}
\caption{Differences in ruin probabilities using two claim distributions}
\label{Fig:dif}
\end{figure}

Figure \ref{Fig:rpexp} displays the two cases for Exponential claims while Figure \ref{Fig:rppar} does that for Pareto claims. All of these four graphs demonstrate a decreasing trend for simulated finite-time ruin probabilities over the amount of initial surplus, which is as expected. In general, the differences between ruin probabilities for Exponentially distributed claim costs and those for Pareto ones are not significant. To be more precise, the exact values of these disparities are plotted in Figure \ref{Fig:dif}. The color bar shows the scale of the graph, and yellow represents values around 0. Indeed, the differences are very small. Furthermore, it can be seen that the disparities behave differently when $\lambda_1<\lambda_2$ and when $\lambda_1>\lambda_2$. For the former case, ruin probabilities for Pareto claims tend to be smaller than those for Exponential claims when the initial reserve is not little, whereas there seems to be no distinction between the two claim distributions in the latter case. One way to explain this is that claim distributions would have more impact on the deficit at ruin because the claim frequency is not affected, the same as in an infinite-time ruin case. However, this is just a sample simulated result from which we cannot draw a general conclusion.\\

On the other hand, it could be seen from the projections on the $y-z$ plane that the magnitude of $\lambda_1$ and $\lambda_2$ causes different monotonicity of ruin probabilities with respect to the fixed window $\xi$. If $\lambda_1>\lambda_2$, the probability of ruin is monotonically increasing with the increase of $\xi$. If $\lambda_1<\lambda_2$, it appears to be the opposite monotonicity. This conclusion for monotonicity is true for both models with heavy-tailed claims and those with light-tailed ones. Such behaviour could also be theoretically verified if we look at the stationary distribution of the Markov Chain created by the exchange of inter claim times given by \eqref{pi1} and \eqref{pi2}. The increase of $\xi$ will raise the probability of getting an inter-claim time smaller than $\xi$ at steady state, i.e., $$\xi \uparrow\;\;\Rightarrow\;\;\pi_1\uparrow, \; \pi_2\downarrow.$$
Then that directly leads to an increasing number of $\tau$. The ruin probability is associated with $$S_T=\sum^{N_1(T)+N_2(T)}_{k=1} Y_k-\sum^{N_1(T)}_{i=1}\tau-\sum^{N_2(T)}_{j=1}\tilde{\tau}$$ for any fixed time $T$, where $N_1(T)$ and $N_2(T)$ denote the number of times $\tau$ and $\tilde{\tau}$ appearing in the process. Notice that $\sum^{N_1(T)}_{i=1}\tau+\sum^{N_2(T)}_{j=1}\tilde{\tau}=T$ stays the same even though the value of $\xi$ alters. So now the magnitude of $S_T$ depends only on $N_1(T)+N_2(T)$ and the distribution of i.i.d $Y_k$. The change of $\xi$ alters only the former value. Intuitively, a rise in $\pi_1$ indicates an increase in $N_1(T)$ and a decrease in $N_2(T)$ whose amount is denoted by $\Delta N_1$ and $\Delta N_2$, respectively. Since the sum of $\tau$s and $\tilde{\tau}$s is kept constant, we have
\begin{eqnarray*}
|\Delta N_1|\mathbb{E}[\tau]&=&|\Delta N_2|\mathbb{E}[\tilde{\tau}]\\
\left|\frac{\Delta N_1}{\Delta N_2}\right|& =& \frac{\mathbb{E}[\tilde{\tau}]}{\mathbb{E}[\tau]}
\end{eqnarray*}
If $\lambda_1>\lambda_2$,  then $\mathbb{E}[\tau]<\mathbb{E}[\tilde{\tau}]$, which implies $\left|\frac{\Delta N_1}{\Delta N_2}\right|>1$. That is to say, the increase of $N_1(T)$ is more than the drop in $N_2(T)$ so that $N_1(T)+N_2(T)$ sees a rise in the end. Thus, it leads to a higher ruin probability. On the contrary, when $\lambda_1<\lambda_2$, i.e., $\mathbb{E}[\tau]>\mathbb{E}[\tilde{\tau}]$, as $\xi$ goes up, ruin probabilities would experience a monotone decay. This reasoning is visually reflected in Figure \ref{Fig:rpexp}-\ref{Fig:rppar} shown above and it could also be noticed that the distribution of claims does not affect such monotonicity.\\

Therefore, by observation, these results suggest that when $\lambda_1<\lambda_2$, the larger choice of the fixed window $\xi$, the smaller the ruin probability will be, and vice versa. On the contrary, when $\lambda_1>\lambda_2$, the larger choice of the fixed window $\xi$, the larger the ruin probability will be, and vice versa. In fact $\lambda_1<\lambda_2$ was mentioned in the introduction (Figure \ref{fig:BMS}) to be an assumption for a Bonus system. Such observation suggests that if the insurer opts to investigate claims histories less frequently, i.e., choosing a larger $\xi$, the ruin probability tends to be smaller. This potentially implies a smaller ruin probability if no premium discount is offered to policyholders. It seems that to minimise an insurer's probability of ruin probably relies more on premium incomes. The use of Bonus systems may not help in decreasing such probabilities. The case of $\lambda_1>\lambda_2$ could be referred to as a Malus system which is unusual in the real world which leads to an opposite conclusion to the other case. This again addresses the significance of premium income to an insurer. In a system with purely maluses, the ruin probability could be reduced if the insurer reviews the policyholders' behaviours more frequently indicating more premium incomes.

\section{Conclusion}
In this paper, we found that a simple Bonus system could be reflected by a dependence structure embedded in a risk model. For the simplest case, we made inter-arrival times switch between two random variables by comparing them with a fixed window $\xi$. Such interchange was equivalently converted from the change of premium rates based on recent claims as shown by Figure \ref{fig:BMS} emulating a basic no claim discount (NCD) system where there are only two classes - either a base or discounted level. Theoretically speaking, it also works for a merely Malus system. Yet in practice, such system does not exist as it probably sounds more tempting if an insurance company offers rewards rather than a penalty.\\

Several different approaches have been undertaken to study the ruin probability under the framework of a regenerative process. It is not surprising under the Cram\'{e}r assumption, the ruin function still has an exponential tail. Since asymptotic results are not exact, we conducted numerical analyses based on different approaches. As a main contribution, we explained how we could construct a discrete Markov additive process from the model under concern when everything is exponentially distributed. By a change of measure via exponential families, ruin probabilities were possible to be simulated through a better presented form \eqref{eqn:mapruin}. Furthermore, we attached a case study using Monte Carlo simulations. It has been discovered that the underlying probability has opposite monotonicity with respect to the fixed time window $\xi$ when two random variables for the inter claim times swap parameters. Additionally, it implies that the use of Bonus systems may not be helpful in reducing ruin probabilities as the premium incomes seem to be more important. However, Bonus systems could still be used as a means of attracting market share which is beneficial to the business in many ways.

\vspace{6pt}


\acknowledgments{This work is under the support from the RARE-318984 project, a Marie Curie IRSES Fellowship within the 7th European Community Framework Programme. Z. Palmowski kindly acknowledges the support from the National Science Centre under the grant 2015/17/B/ST1/01102.}

\authorcontributions{Z.P. and S.D. conceived and designed the model; W.N. connected the work with the Bonus-Malus model; Z.P., W.N. and C.C. conducted analyses in details;  S.D. introduced simulation methods;  C.C. and W.N. interpreted the idea and wrote the paper. All authors substantially contributed to the work.}

\conflictofinterests{The authors declare no conflict of interest. The founding sponsors had no role in the design of the study; in the collection, analyses, or interpretation of data; in the writing of the manuscript, and in the decision to publish the results.}

%


\bibliographystyle{mdpi}

\renewcommand\bibname{References}


\begin{thebibliography}{999}
\bibitem{Alfredo1}
Afonso, L.B., Eg\'{i}dio dos Reis, A.D. and Waters, H.R., Calculating
continuous time ruin probabilities for a large portfolio with varying premiums.
{\it ASTIN Bulletin}, {\bf 2009}, {\it 39(01)}, 117--136.

\bibitem{Alfredo2}
Afonso, L.B., Rui, C., Eg\'{i}dio dos Reis, A.D., and Guerreiro, G.,
Bonus malus systems and finite and continuous time ruin probabilities
in motor insurance. {\bf 2015}, {\it Preprint}.

\bibitem{albrecher2004ruin}
Albrecher, H. and Boxma, O.J., A ruin model with dependence between claim sizes and claim intervals. {\it Insurance: Mathematics and Economics}, {\bf 2004}, {\it 35(2)}, 245--254.

\bibitem{albrecher2011explicit}
Albrecher, H., Constantinescu, C. and Loisel, S., Explicit ruin formulas for models with dependence among risks. {\it Insurance: Mathematics and Economics}, {\bf 2011}, {\it 48(2)}, 265--270.

\bibitem{asmussen2014}
Asmussen, S., Modeling and Performance of Bonus-Malus Systems: Stationarity versus Age-Correction. {\it Risks}, {\bf 2014}, {\it 2(1)}, 49--73.

\bibitem{RP}
Asmussen, S. and Albrecher, H., {\it Ruin Probabilities}, 2010, Vol. 14. World Scientific.

\bibitem{Jose}
Blanchet, J. and Glynn, P.,  Efficient rare-event simulation for the maximum of heavy-tailed random walks, {\it Ann. Appl. Probab.} {\bf 2008}, {\it 18(4)}, 1351--1378.

\bibitem{Dubey}
Dubey, A., Probabilit{\'{e}} de ruine lorsque le param{\`{e}}tre de Poisson est ajust{\'{e}} a posteriori, {\it Mitteilungen der Vereinigung schweiz Versicherungsmathematiker}, {\bf 1977}, {\it 2}, 130--141.

\bibitem{embrechts1997}
Embrechts, P. and Kl\:{u}ppelberg, C. and Mikosch, T.,  {\it Modelling extremal events for insurance and finance,} Berlin, 1997, Springer, c1997.



\bibitem{FKZ}
 Foss, S., Korshunov, D. and Zachary S., {\it An Introduction to Heavy-tailed and Subexponential Distributions,} 2011, Springer.

\bibitem{goldie}
Goldie, C.M., Implicit renewal theory and tails of solutions of random equations. {\it The Annals of Applied Probability}, {\bf 1991}, 126--166.

\bibitem{kwan2010dependent}
Kwan, I.K. and Yang, H., Dependent insurance risk model: deterministic threshold. {\it Communications in Statistics—Theory and Methods}, {\bf 2010}, {\it 39(5)}, 765--776.

\bibitem{Lemaire}
Lemaire, J., {\it Bonus-malus systems in automobile insurance}. 2012, Springer science \& business media.

\bibitem{Li2015}
Li, B. and Ni, W. and Constantinescu, C., Risk models with premiums adjusted to claims number. {\it Insurance: Mathematics and Economics}, {\bf 2015}, {\it 65(2015)}, 94--102.


\bibitem{ni2014}
Ni, W., Constantinescu, C. and Pantelous, A.A.,  Bonus–Malus systems with Weibull distributed claim severities. {\it Annals of Actuarial Science}, {\bf 2014}, {\it 8(02)}, 217--233.

\bibitem{BP1}
Palmowski, Z. and Zwart, B., Tail asymptotics for the supremum of a regenerative process,
{\it J. Appl. Probab.} {\bf 2007}, {\it 44(2)}, 349--365.

\bibitem{BP2}
Palmowski, Z. and Zwart, B., On perturbed random walks,
{\it J. Appl. Probab.} {\bf 2010}, {\it 47(4)}, 1203--1204.

\bibitem{Rolski} Rolski, T., Schmidli, H., Schmidt, V. and Teugles,
J.L., {\it Stochastic processes for insurance and finance.}
John Wiley and Sons, Inc., New York, 1999.

\bibitem{valdez2002ruin}
Valdez, E.A. and Mo, K., Ruin probabilities with dependent claims. {\it Working paper}, UNSW, Sydney, Australia, 2002.
\end{thebibliography}


\end{document}